\documentclass{amsart}
\usepackage{comment}
\usepackage {amssymb}
\usepackage{amscd}
\usepackage{amsmath, amssymb, amsgen, amsthm, amscd, xspace}
\renewcommand{\k}{{\mathfrak{k}}}
\renewcommand{\t}{{\mathfrak{t}}}

\newcommand{\Kt}{\widetilde{K}}

\def\Ext{\mathop{\hbox {Ext}}\nolimits}

\def\dim{\mathop{\hbox {dim}}\nolimits}
\def\Ad{\mathop{\hbox {Ad}}\nolimits}
\def\ad{\mathop{\hbox {ad}}\nolimits}

\def\ch{\mathop{\hbox {ch}}\nolimits}

\def\det{\mathop{\hbox {det}}\nolimits}

\def\ell{\mathop{\hbox {ell}}\nolimits}

\def\Hom{\mathop{\hbox {Hom}}\nolimits}
\def\im{\mathop{\hbox {Im}}\nolimits}
\def\Im{\mathop{\hbox {Im}}\nolimits}

\def\Ker{\mathop{\hbox{Ker}}\nolimits}

\newcommand{\pf}{\begin{proof}}
\newcommand{\epf}{\end{proof}}
\newcommand{\eq}{\begin{equation}}
\newcommand{\eeq}{\end{equation}}
\newcommand{\eqn}{\begin{equation*}}
\newcommand{\eeqn}{\end{equation*}}

\newcommand{\fra}{\mathfrak{a}}
\newcommand{\frb}{\mathfrak{b}}

\newcommand{\frg}{\mathfrak{g}}
\newcommand{\frh}{\mathfrak{h}}
\newcommand{\frk}{\mathfrak{k}}
\newcommand{\frl}{\mathfrak{l}}

\newcommand{\frp}{\mathfrak{p}}
\newcommand{\frq}{\mathfrak{q}}
\newcommand{\frr}{\mathfrak{r}}
\newcommand{\frs}{\mathfrak{s}}
\newcommand{\frt}{\mathfrak{t}}
\newcommand{\fru}{\mathfrak{u}}

\newcommand{\frso}{\mathfrak{so}}

\newcommand{\caO}{\mathcal{O}}

\newcommand{\caS}{\mathcal{S}}

\newcommand{\bbC}{\mathbb{C}}
\newcommand{\bbH}{\mathbb{H}}
\newcommand{\bbN}{\mathbb{N}}

\newcommand{\bbR}{\mathbb{R}}
\newcommand{\R}{\mathbb{R}}

\newcommand{\bbZ}{\mathbb{Z}}
\newcommand{\tr}{\operatorname{tr}}

\newtheorem{theorem}[equation]{Theorem}
\newtheorem{cor}[equation]{Corollary}
\newtheorem{prop}[equation]{Proposition}
\newtheorem{lemma}[equation]{Lemma}

\theoremstyle{remark}
\newtheorem{remark}[equation]{Remark}

\theoremstyle{definition}
\newtheorem{definition}[equation]{Definition}
\newtheorem{example}[equation]{Example}
\newtheorem{conj}[equation]{Conjecture}
\numberwithin{equation}{section}
\allowdisplaybreaks[1]
\begin{document}

%
%
\title{Dirac cohomology, elliptic representations and endoscopy}
\author{Jing-Song Huang}
\address{Department of Mathematics, Hong Kong University of Science and
technology,
Clear Water Bay, Kowloon, Hong Kong SAR, China}
\email{mahuang@ust.hk}
\thanks{The research described in this paper is supported by grants from
Research Grant Council of HKSAR and National Science Foundation of China.
This paper will appear in `Representations of Reductive Groups, in Honor of 60th Birthday of David Vogan', edited by M. Nervins
and P. Trapa, published by Springer.}

\dedicatory{To David Vogan for his
60th birthday}

%
%
\begin{abstract} The first part (Sections 1-6) of this paper is a
survey of some of the recent developments
in the theory of Dirac cohomology,
especially the relationship of Dirac cohomology with $(\frg,K)$-cohomology and nilpotent Lie algebra cohomology;
the second part (Sections 7-12) is devoted to understanding the
unitary elliptic representations and endoscopic
transfer by using the techniques in Dirac cohomology.  A few problems
and conjectures are proposed for further investigations.
\end{abstract}
\keywords{Dirac cohomology, Harish-Chandra module, elliptic representation,
pseudo-coefficient, endoscopy}
\subjclass[2010]{22E46, 22E47}
\maketitle     
%
%
\section*{Introduction}
%

Since its appearance in the literature \cite{HP1}, Dirac cohomology has been playing an active role in
many of the recent developments in representation theory.  Back in late 1990s, Vogan made a conjecture
on the property of Dirac operator in the setting of a reductive
Lie algebra and its associated Clifford algebra \cite{V3}.  This property implies that the
standard parameter of the infinitesimal character of a Harish-Chandra module
$X$ and the infinitesimal character of its Dirac cohomology $H_D(X)$ are conjugate under
the Weyl group.
Vogan's conjecture was consequently verified in \cite{HP1}, and it has been extended
to several other settings by many authors (see the remark at the end of Section 1).

Dirac cohomology of various classes of representations is intimately related
to several classical subjects of representation theory like global
characters and geometric construction of the discrete series (see [HP2]).
The Dirac cohomology of several families of Harish-Chandra modules has been
determined. These modules include finite-dimensional modules and
irreducible unitary $A_\frq(\lambda)$-modules [HKP].  It was proved that
if $X$ is a unitary Harish-Chandra modules then
$$H^*(\frg,K;X\otimes F^*)\cong \Hom(H_D(F),H_D(X))$$
for any irreducible finite-dimensional module $F$. It is evident that
unitary representations with nonzero Dirac cohomology are closely related to automorphic representations.
In [HP2] we used Dirac cohomology to extend the Langlands formula on dimensions of automorphic
forms [L1] to a slightly more general setting.

Another aspect of Dirac cohomology is its connection with $\fru$-cohomology.
In particular,  when $G$ is Hermitian symmetric and $\fru$ is unipotent radical
of a parabolic subalgebra with Levi subgroup $K$,
[HPR] showed that for a unitary representation its Dirac cohomology
 is isomorphic to its $\fru$-cohomology
up to a twist of a one-dimensional character.  In particular, Enright's
calculation of $\fru$-cohomology [E] gives the Dirac
cohomology of the irreducible unitary highest weight modules.
The Dirac cohomology of unitary lowest weight modules of scalar type is calculated
more explicitly in [HPP].
The Euler characteristic of Dirac cohomology gives the
K-character of the Harish-Chandra module.  As an application, we generalized the classical
theorem of Littlewood on branching rules in [HPZ] and some of the other classical branching rules in [H2].

Kostant extended Vogan's conjecture
to the setting of the cubic Dirac operator and proved a non-vanishing
result on Dirac cohomology for highest weight modules in the most general
setting \cite{Ko3}.   He also determined the Dirac cohomology of
finite-dimensional modules in the equal rank case.
The Dirac cohomology for all irreducible highest weight modules
was determined in [HX] in terms of coefficients of Kazhdan-Lusztig
polynomials.  The general formula relating the Dirac cohomology and $\fru$-cohomology
for irreducible highest weight modules is also proven in [HX].

The aim of this paper is twofold:  First, we review some of the recent developments of
Dirac cohomology, in particular its relationship with $(\frg,K)$-cohomology
and $\fru$-cohomology.  Second, we use Dirac cohomology
as a tool to study a class of irreducible unitary representations, called elliptic
representations.
Harish-Chandra showed that the characters of irreducible or more generally
admissible representations are locally integrable functions and smooth
on the open dense subset of regular elements [HC1]. An elliptic representation
has a global character that does not vanish on the elliptic elements in the set of regular elements.
For real reductive Lie groups with compact Cartan subgroups, the irreducible tempered
elliptic  representations  are showed to be representations with nonzero
Dirac cohomology, and they are precisely the discrete series and some
of the limits of discrete series.
The characters of the irreducible tempered elliptic representations are
associated in a natural way to the supertempered distributions
defined by Harish-Chandra [HC4].
We conjecture that in general the elliptic unitary representations are precisely
the unitary representations with nonzero Dirac cohomology.

We show that an irreducible admissible
(not necessarily unitary) representation
is elliptic if and only if its Dirac index is not zero.  We prove that under the condition of
regular infinitesimal character, the Dirac index is zero if and only if the
Dirac cohomology is zero.  We conjecture that this equivalence holds in general without
the regularity condition.  We also show that the Harish-Chandra modules of
irreducible elliptic unitary representations with regular infinitesimal characters are
$A_\frq(\lambda)$-modules
for a real reductive algebraic group $G(\R)$.

We also observe a connection between the Labesse's calculation of the
endoscopic transfer of pseudo-coefficients of discrete series and
the calculation of the characters of Dirac index of discrete series. It offers
a new point of view for understanding the endoscopic transfer in the framework of
the Dirac cohomology and Dirac index.  To classify the irreducible unitary representations with
nonzero Dirac cohomology remains to be an open and interesting
problem.  We conjecture at the end of the paper that any irreducible
unitary representation which does not have nonzero Dirac cohomology
is induced from those with nonzero Dirac cohomology.

\section{Vogan's conjecture on Dirac cohomology}
%

For a real reductive group $G$ with a Cartan involution $\theta$, denote by $\frg_0$ its Lie algebra
and assume that $K=G^\theta$ is a maximal compact subgroup of $G$.
Let $\frg=\frk\oplus\frp$ be the Cartan decomposition
for the complexified Lie algebra of $G$.  Let $B$ be a non-degenerate invariant symmetric bilinear
form on $\frg$, which restricts to the Killing form on the semisimple part $[\frg,\frg]$ of $\frg$.

Let $U(\frg)$ be the universal enveloping
algebra of $\frg$ and $C(\frp)$ the Clifford algebra of
$\frp$ with respect to $B$. Then one can consider the following version of the Dirac
operator:
$$
D=\sum_{i=1}^n Z_i\otimes Z_i \in U(\frg)\otimes C(\frp);
$$
here $Z_1,\dots,Z_n$ is an orthonormal basis of $\frp$ with respect to the symmetric bilinear
form $B$.  It follows that $D$ is independent of the choice of the orthonomal basis $Z_1,\dots,Z_n$
and it is invariant under the diagonal adjoint action of $K$.

The Dirac operator $D$ is a square root of Laplace operator
associated to the symmetric pair $(\frg,\frk)$. To explain this, we
start with a Lie algebra map
\begin{equation*}
\alpha:\frk\rightarrow C(\frp)
\end{equation*}
which is defined by the adjoint map $\ad:\frk\rightarrow\frso(\frp)$
composed with the embedding of $\frso(\frp)$ into $C(\frp)$ using
the identification $\frso(\frp)\simeq\bigwedge^2\frp$. The explicit
formula for $\alpha$ is (see \S2.3.3 [HP2])
\begin{equation}
\alpha(X)=-\frac{1}{4}\sum_{j}[X,Z_j]Z_j.
\end{equation}
Using $\alpha$ we can embed the Lie algebra $\frk$ diagonally into
$U(\frg)\otimes C(\frp)$, by
\begin{equation*}\label{Delta map}
X\mapsto X_\Delta=X\otimes1+1\otimes\alpha(X).
\end{equation*}
This embedding extends to $U(\frk)$. We denote the image of $\frk$
by $\frk_\Delta$, and then the image of $U(\frk)$ is the enveloping
algebra $U(\frk_\Delta)$ of $\frk_\Delta$.

Let $\Omega_\frak g$ be the Casimir operator
for $\frak g$, given by $\Omega_\frak g = \sum Z_i^2 -\sum W_j^2$,
where $W_j$ is an orthonormal basis for $\frak k_0$ with respect to
the inner product $-B$, where $B$ is the Killing form. Let
$\Omega_{\frk}=-\sum W_j^2$ be the Casimir operator for $\frak k$.
The image of $\Omega_\frk$ under $\Delta$ is denoted by
$\Omega_{\frk_\Delta}$.

Then
\begin{equation}\label{D^2}
D^2 = -\Omega_\frak g\otimes 1 + \Omega_{\frak k_\Delta} +
(||\rho_c||^2-||\rho||^2)1\otimes 1,
\end{equation}
where $\rho$ and $\rho_c$ are half sums of positive roots and compact
positive roots respectively.

The Vogan conjecture says that every element $z\otimes 1$ of
$Z(\frg)\otimes 1 \subset U(\frg)\otimes C(\frp)$ can be written as
$$\zeta(z)+Da+bD$$
where $\zeta(z)$ is in $Z(\k_\Delta)$,
and $a,b\in U(\frg)\otimes C(\frp)$.

A main result in [HP1]  is  introducing a differential $d$
on the $K$-invariants in $U(\frg)\otimes C(\frp)$ defined by a super bracket with $D$,
and determination of the cohomology of this
differential complex.  As a consequence,
Pand\v{z}i\'c and I proved the following theorem.
In the following we denote by $\frh$ a Cartan subalgebra of $\frg$
containing a Cartan subalgebra $\frt$ of $\frk$ so that $\frt^*$ is embedded into $\frh^*$,
and by $W$ and $W_K$ the Weyl groups of $(\frg,\frh)$ and $(\frk,\frt)$
respectively.

\begin{theorem}[\cite{HP1}]
Let $\zeta:
Z(\frg)\rightarrow Z(\frk)\cong Z(\frk_{\Delta})$ be the algebra homomorphism
that is determined
by the following commutative diagram:
\begin{equation*}\label{Vogan's diagram}
\CD
  Z(\mathfrak{g}) @> \zeta >> Z(\mathfrak{k}) \\
  @V \eta VV @V \eta_{\mathfrak{k}} VV  \\
  P(\mathfrak{h}^*)^{W} @>\mathrm{Res}>>
  P(\mathfrak{t}^*)^{W_K},
\endCD
\end{equation*}
where $P$ denotes the polynomial algebra, and  vertical maps $\eta$ and $\eta_\frk$  are Harish-Chandra
isomorphisms.
Then for each $z\in Z(\frg)$ one has
\begin{equation*}\label{Vogan's equation}
z\otimes 1-\zeta(z)=Da+aD, \text{\ for some\ }a\in U(\frg)\otimes C(\frp).
\end{equation*}

\end{theorem}

For any admissible $(\frg,K)$-module $X$,
Vogan ([V3], [HP1]) introduced the notion of \emph{Dirac cohomology} $H_D(X)$ of $X$.
Consider  the action of the Dirac operator $D$  on $X\otimes S$,
with $S$ the spinor module  for the Clifford algebra $C(\frp)$. The Dirac cohomology is defined as follows:
$$H_D(X)\colon =\Ker D/ \Im D \cap \Ker D.$$
It follows from the identity (\ref{D^2}) that $H_D(X)$ is a finite-dimensional module
for the spin double cover $\Kt$ of $K$. In case $X$ is unitary, $H_D(X)=\ker D=\ker D^2$
since $D$ is self-adjoint with respect to a natural Hermitian inner product on $X\otimes S$.
As a consequence of the above theorem, we have that
 $H_D(X)$, if nonzero, determines the infinitesimal character of $X$.

\begin{theorem}[\cite{HP1}]
Let $X$ be an admissible $(\frg,K)$-module with standard infinitesimal character parameter
$\Lambda\in \frh^*$.  Suppose that $H_D(X)$ contains a representation of
$\Kt$ with infinitesimal character $\lambda$.  Then
$\Lambda$ and $\lambda\in\frt^*\subseteq \frh^*$ are conjugate under $W$.
\end{theorem}

The above theorem is proved in [HP1] for a connected semisimple
Lie group $G$.  It is straightforward to extend the result to a
possibly disconnected reductive Lie group  in Harish-Chandra's class [DH2].

Vogan's conjecture implies a refinement of the celebrated
Parthasarathy's Dirac inequality, which is an extremely useful tool
for the classification of irreducible unitary representations of
reductive Lie groups.

\begin{theorem}[Extended Dirac Inequality \cite{P}, \cite{HP1}]
Let $X$ be an irreducible unitary $(\frg,K)$-module with infinitesimal
character $\Lambda$.  Fix a representation of $K$ occurring in $X$
with a highest weight $\mu\in\t^*$, and a positive root system $\Delta^+(\frg)$
for $\frt$ in $\frg$.  Here $\t$ is a Cartan subalgebra of $\k$.
Write
$$\rho_c=\rho(\Delta^+(\frk)),\ \rho_n=\rho(\Delta^+(\frp)).$$
Fix an element $w\in W_K$ such that $w(\mu-\rho_n)$ is dominant
for $\Delta^+(\frk)$. Then
$$\langle w(\mu-\rho_n)+\rho_c,w(\mu-\rho_n)+\rho_c\rangle
\geq \langle \Lambda,\Lambda\rangle.$$
The equality  holds if and only if some $W$
conjugate of
$\Lambda$ is equal to
$w(\mu-\rho_n)+\rho_c$.
\end{theorem}

\begin{remark} Dirac cohomology becomes a useful tool in representation theory
 and related areas with Vogan's conjecture being extended to various different settings,
 most notably Kostant's generalization to the setting of the cubic Dirac operator,
 which will be discussed in detail in Section 4. We also mention the following
 extensions.
\begin{item}
\item{(i)} Alekseev and Meinrenken have proved a version of Vogan's conjecture in
their study of `Lie theory and the Chern-Weil
homomorphism' \cite{AM}.

\item{(ii)} Kumar has proved a similar version of Vogan's conjecture
in `Induction functor in non-commutative equivariant
cohomology and Dirac cohomology' \cite{Ku}.

\item{(iii)} Pand\v{z}i\'c and I have extended the Vogan's conjecture
to  the symplectic Dirac operator in Lie
superalgebras \cite{HP3}.

\item{(iv)} Kac, M\"oseneder Frajria and Papi have extended the Vogan's conjecture to the affine cubic Dirac operator in the
affine Lie algebras [KMP].

\item{(v)} Barbasch, Ciubotaru and Trapa have extended the Vogan's conjecture to the setting of the graded
affine Hecke algebras [BCT].

\item{(vi)} Ciubotaru and Trapa have proved a version of the Vogan's conjecture for studying Weyl group representations in connection with the Springer theory [CT].

\end{item}
\end{remark}

\section{Dirac cohomology of Harish-Chandra modules}
%
We now describe the Dirac cohomology of finite-dimensional modules and irreducible unitary representations with strongly regular infinitesimal characters, which are $A_\frq(\lambda)$-modules.
These results are proved in [HKP].

Recall that $\frt_0$ is a Cartan
subalgebra of $\frk_0$ and $\frh_0\supseteq \frt_0 $ is a
fundamental Cartan subalgebra of $\frg_0$. Then  $\frh_0=\frt_0\oplus\fra_0$ with $\fra_0$ the centralizer of $\frt_0$ in $\frp_0$. Passing to complexifications,
we will view $\frt^*$ as a subspace of $\frh^*$ by extending the functionals to act as 0 on $\fra$.

We denote by $\Delta(\frg,\frh)$ (respectively $\Delta(\frg,\frt)$) the root system of $\frg$ with respect
to $\frh$ (respectively $\frt$). The root system of $\frk$ with respect to $\frt$ will be denoted by
$\Delta(\frk,\frt)$. Note that $\Delta(\frg,\frh)$ and $\Delta(\frk,\frt)$ are reduced, while
$\Delta(\frg,\frt)$ is in general not reduced.
The Weyl groups corresponding to the above root systems are denoted by
$$W=W(\frg,\frh), W(\frg,\frt), \text{\ and\ }
W_K=W(\frk,\frt).$$

Throughout this section we fix compatible choices of positive roots $\Delta^+(\frg,\frh)$,
$\Delta^+(\frg,\frt)$ and $\Delta^+(\frk,\frt)$. As usual, we denote by $\rho$ the half sum of positive roots
for $(\frg,\frh)$, by $\rho_c$ the half sum of positive roots for $(\frk,\frt)$, and by $\rho_n$ the difference
$\rho-\rho_c$. Then $\rho,\rho_c,\rho_n\in\frt^*$.

We let $\frt^*_{\bbR}=i\frt^*_0$ and let $\frh^*_{\bbR}=i\frt^*_0+\fra^*_0$.
Our fixed form $B$ on $\frg$ induces inner products on $\frt^*_{\bbR}$ and $\frh^*_{\bbR}$.

We denote by $C_{\frg}(\frh^*_{\bbR})$ (respectively $C_{\frg}(\frt^*_{\bbR})$, $C_{\frk}(\frt^*_{\bbR})$) the
closed Weyl chamber corresponding to $\Delta^+(\frg,\frh)$ (respectively $\Delta^+(\frg,\frt)$,
$\Delta^+(\frk,\frt)$). Then $C_{\frg}(\frt^*_{\bbR})$ is contained in $C_{\frg}(\frh^*_{\bbR})$. Namely, if
$\mu\in\frt^*_{\bbR}\subset\frh^*_{\bbR}$ has nonnegative inner product with every element of
$\Delta^+(\frg,\frt)$, then for any $\alpha\in \Delta^+(\frg,\frh)$
$$\langle \mu,\alpha\rangle =\langle \mu,\alpha|_\frt\rangle +\langle\mu,
\alpha|_\fra\rangle \geq 0,$$ because $\mu$ is orthogonal to $\fra^*$.

We define
$$
W(\frg,\frt)^1=\{w\in W(\frg,\frt)\ | w(C_{\frg}(\frt^*_{\bbR}))\subset C_{\frk}(\frt^*_{\bbR})\}.
$$
It is clear that $W(\frk,\frt)$ is a subgroup of $W(\frg,\frt)$, and that the multiplication map induces a
bijection from $W(\frk,\frt)\times W(\frg,\frt)^1$ onto $W(\frg,\frt)$. Thus the set $W(\frg,\frt)^1$ is in
bijection with $W(\frk,\frt)\backslash W(\frg,\frt)$.
Let $E_\mu$ denote the irreducible representation of $\frk$ with highest weight $\mu$.
 The following fact can be found in
[P] (see also Chapter II Lemma 6.9 \cite{BW}, or Lemma 9.3.2 \cite{Wa1}):

\begin{lemma} \label{spinhwts}
We have the following isomorphism for $\frk$-modules:
$$
S\cong \bigoplus_{w\in W(\frg,\frt)^1}2^{[l_0/2]}E_{w\rho-\rho_c},
$$
where $l_0=\dim\fra$ and $mE_\mu$ means a direct sum of $m$ copies of $E_\mu$.
\end{lemma}
Clearly,  $S$ is isomorphic to the Dirac cohomology $H_D(\bbC)$ of the trivial representation $\bbC$.

Let $V_\lambda$ be the irreducible finite-dimensional $(\frg,K)$-module with highest weight $\lambda\in\frh^*$.
If Dirac cohomology of $V_\lambda$ is nonzero, then $\lambda+\rho\in\frt^*$ and thus $\lambda\in\frt^*$.
We have to identify highest weights $\gamma$ of $\tilde K$-submodules of
$V_\lambda\otimes S$ which satisfy $\|\gamma+\rho_c\|=\|\lambda+\rho\|$.

\begin{theorem}[Theorem 4.2 \cite{HKP}]
Let $V_\lambda$ be an irreducible finite-dimensional $(\frg,K)$-module with highest weight $\lambda$. If
$\lambda\neq\theta\lambda$ then the Dirac cohomology of $V_\lambda$ is zero. If $\lambda=\theta\lambda$, then as
a $\frk$ module the Dirac cohomology of $V_\lambda$ is:
$$
H_D(V_\lambda)=\bigoplus_{w\in W(\frg,\frt)^1}2^{[l_0/2]}E_{w(\lambda+\rho)-\rho_c}.
$$
\end{theorem}

We now describe the Dirac cohomology of a unitary $A_\frq(\lambda)$-module.
Recall that a $\theta$-stable parabolic subalgebra
$$
\frq=\frl\oplus\fru
$$
of $\frg$ is by definition the sum of nonnegative eigenspaces of $\ad(H)$, where $H$ is some fixed element of
$i\frt_0$ (and consequently $\ad(H)$ is semisimple with real eigenvalues).
 The Levi
subalgebra $\frl$ of $\frq$ is the zero eigenspace of $\ad(H)$, while the nilradical $\fru$ of $\frq$ is the sum
of positive eigenspaces of $\ad(H)$. Note that clearly $\frl\supseteq\frh$. Since $\theta(H)=H$, $\frl,\fru$ and
$\frq$ are all invariant under $\theta$. Furthermore, $\frl$ is real, i.e., $\frl$ is the complexification of a
subalgebra $\frl_0$ of $\frg_0$. Let $L$ denote the connected subgroup of $G$ corresponding to $\frl_0$. We will
assume that our fixed choice of positive roots $\Delta^+(\frg,\frh)$ is compatible with $\frq$ in the sense that
the set of roots
$$
\Delta(\fru)=\{\alpha\in\Delta(\frg,\frh) | \frg_\alpha\subset\fru\}
$$
is contained in $\Delta^+(\frg,\frh)$. Note that $\Delta(\frl,\frh)\subseteq\Delta(\frg,\frh)$, and we set
$\Delta^+(\frl,\frh)=\Delta(\frl,\frh)\cap\Delta^+(\frg,\frh)$. Likewise,
$\Delta(\frl,\frt)\subseteq\Delta(\frg,\frt)$, and we set
$\Delta^+(\frl,\frt)=\Delta(\frl,\frt)\cap\Delta^+(\frg,\frt)$.

Let $\lambda\in\frl^*$ be admissible. In other words, $\lambda$ is the complexified differential of a unitary
character of $L$,  satisfying the following positivity condition:
$$
\langle\alpha,\lambda|_\frt\rangle\geq 0,\qquad\text{for all}\quad \alpha\in\Delta(\fru).
$$
Then $\lambda$ is orthogonal to all roots of $\frl$, so we can view $\lambda$ as an element of $\frh^*$.

Given $\frq$ and $\lambda$ as above, define
$$
\mu(\frq,\lambda)=\lambda|_\frt+2\rho(\fru\cap\frp).
$$
Here $\rho(\fru\cap\frp)=\rho(\Delta(\fru\cap\frp))$ is the half sum of all elements of $\Delta(\fru\cap\frp)$,
i.e., of all $\frt$-weights of $\fru\cap\frp$, counted with multiplicity. We will use analogous notation for
other $\frt$-stable subspaces of $\frg$.

The following result of Vogan and Zuckerman characterizes the $A_\frq(\lambda)$ modules we wish to consider.

\begin{theorem}[\cite{VZ},\cite{V2}]
Let $\frq$ be a $\theta$-stable parabolic subalgebra of $\frg$ and let $\lambda\in\frh^*$ be admissible as
defined above. Then there is a unique unitary $(\frg,K)$-module $A_\frq(\lambda)$ with the following properties:

{\bf (i)} The restriction of $A_\frq(\lambda)$ to $\frk$ contains the representation with highest weight
$\mu(\frq,\lambda)$ defined as above;

{\bf (ii)} $A_\frq(\lambda)$ has infinitesimal character $\lambda+\rho$;

{\bf (iii)} If the representation of $\frk$ occurs in $A_\frq(\lambda)$, then its highest weight is of the form
\begin{equation}
\label{eqnkt} \mu(\frq,\lambda)+\sum_{\beta\in\Delta(\fru\cap\frp)}n_\beta\beta
\end{equation}
with $n_\beta$ non-negative integers. In particular, $\mu(\frq,\lambda)$ is the lowest $K$-type of
$A_\frq(\lambda)$ (and its multiplicity is 1.)
\end{theorem}

We denote the Weyl groups for $\Delta(\frl,\frt)$ and $\Delta(\frl,\frh)$ by $W(\frl,\frt)$ and $W(\frl,\frh)$
respectively. Clearly, these are subgroups of $W(\frg,\frt)$, respectively $W(\frg,\frh)$.

\begin{theorem}[Theorem 5.1 \cite{HKP}]
If $\lambda\neq\theta\lambda$, then the Dirac cohomology of $A_\frq(\lambda)$ is zero. If
$\lambda=\theta\lambda$, then the Dirac cohomology of the unitary irreducible $(\frg,K)$-module
$A_\frq(\lambda)$ is:
$$
H_D(A_\frq(\lambda))=\ker D=\bigoplus_{w\in W(\frl,\frt)^1}2^{[l_0/2]}E_{w(\lambda+\rho)-\rho_c}.
$$
\end{theorem}

\begin{remark} Dirac cohomology has been calculated for other families of representations
(see [BP1], [BP2], [MP]).
\end{remark}

\section{Dirac cohomology and $(\frg,K)$-cohomology}
%
 Let $F$ be an irreducible finite-dimensional $G$-module with highest weight $\lambda$.
By results of Vogan and Zuckerman \cite{VZ}, the irreducible unitary  $(\frg,K)$-modules $X$ such that
$H^*(\frg,K;X\otimes F^*)\neq 0$ are certain $A_{\frq}(\lambda)$ modules with the
same infinitesimal character as $F$.  Moreover, if $X$ is such an $A_{\frq}(\lambda)$ module, then
$$
H^i(\frg,K;X\otimes F^*)=\Hom_{L\cap K}(\textstyle{\bigwedge^{i-\dim(\fru\cap\frp)}}(\frl\cap\frp),\Bbb C),
$$
where $L$ is the Levi subgroup of $G$ corresponding to $\frq$.

Recall that the above $(\frg,K)$-cohomology can be defined as the cohomology of the complex
$$
\Hom^\cdot_K(\textstyle{\bigwedge^\cdot}(\frp),X\otimes F^*),
$$
with differential
$$
df(X_1\wedge\dots\wedge X_k)=\sum_i (-1)^{i-1} X_i\cdot f(X_1\wedge\dots\widehat X_i\dots\wedge X_k).
$$

To show how this is related to our results, let us first show that $(\frg,K)$-cohomology is related to Dirac
cohomology, as stated in the introduction. As we mentioned above, if $(\frg,K)$-cohomology is nonzero, then $X$
must have the same infinitesimal character as $F$. We assume this in the following.

Consider first the case when $\dim\frp$ is even. Then we can write $\frp$ as a direct sum of isotropic subspaces
$U$ and $\bar U\cong U^*$. Then we have the spinor spaces $S=\bigwedge^\cdot U$ and $S^*=\bigwedge^\cdot\bar U$,
and
$$
S\otimes S^*\cong\textstyle{\bigwedge^\cdot}(U\oplus\bar U)=\textstyle{\bigwedge^\cdot}\frp.
$$
It follows that we can identify the $(\frg,K)$-cohomology of $X\otimes F^*$ with
$$
H^*(\Hom^\cdot_{\tilde K}(S\otimes S^*,X\otimes F^*))\cong H^*(\Hom^\cdot_{\tilde K}(F\otimes S,X\otimes S)).
$$
If $X$ is unitary, Wallach has proved that the differential of this complex is 0 (see \cite{Wa1}, Proposition
9.4.3, or \cite{BW}). So taking cohomology can be omitted in the above formula. It follows that
$$
H^*(\frg,K;X\otimes F^*) = \Hom^\cdot_{\tilde K}(H_D(F),H_D(X)).
$$
Namely, the eigenvalues of $D^2$ are non-positive on $F\otimes S$ and nonnegative on $X\otimes S$
(see \cite{Wa1}, 9.4.6). Also, since the infinitesimal characters
of $X$ and $F$ are the same, the eigenvalue of $D^2$ on a $\widetilde K$-type in either of the two variables
depends only on the value of the Casimir element $\Omega_{\frk_\Delta}$ on that $\widetilde K$-type. In
particular, action of $D^2$ on isomorphic $\widetilde K$-types must have the same eigenvalue. It follows
from the Dirac inequality that the same
$\widetilde K$-type can appear in both $F\otimes S$ and $X\otimes S$ only if it is in the kernel of $D^2$ in
each variable, and $\Ker D^2$ is equal to the Dirac cohomology for these cases.

Now we consider the case when $\dim\frp$ is odd.  In this case, $\textstyle{\bigwedge^\cdot}\frp$ is isomorphic
to the direct sum of two copies of $S\otimes S^*$. Therefore, $H^*(\frg,K;X\otimes F^*)$ is isomorphic to the
direct sum of two copies of $\Hom^\cdot_{\tilde K}(H_D(F),H_D(X))$.

If we now use the formulas for $H_D(A_\frq(\lambda))$ and $H_D(F)$ from Section 2, we immediately get
$$
\dim H^*(\frg,K;X\otimes F^*)= 2^{l_0}|W(\frl,\frt)/W(\frl\cap \frk,\frt)|.
$$
This agrees with the results of \cite{VZ}.

%
\section{Dirac cohomology of highest weight modules}
%
We describe Dirac cohomology of irreducible highest weight modules.
As mentioned in section 2, Kostant extended Vogan's conjecture
to the setting of the cubic Dirac operator [Ko3]. Fix a Cartan subalgebra
$\frh$ in a Borel subalgebra $\frb$. The category $\caO$ introduced by Bernstein, Gelfand and Gelfand [BGG]
is the category of all $\frg$-modules, which are finitely generated, locally $\frb$-finite
and semisimple under $\frh$-action.  Kostant proved a non-vanishing
result on Dirac cohomology for highest weight modules in the most general
setting.  His theorem implies that for the equal rank case all highest weight modules have non-zero Dirac cohomology. He also determined the Dirac cohomology of finite-dimensional modules
in this case.  The connection of Dirac cohomology of $(\frg,K)$-modules
and that of highest weight modules was studied in [DH1] by using the Jacquet functor.
In [HX] we determined the Dirac cohomology of all irreducible highest weight modules
in terms of Kazhdan-Lusztig polynomials.

We first recall the definition of Kostant's cubic Dirac operator
and the basic properties of the corresponding Dirac cohomology.
Let $\frg$ be a semisimple complex Lie algebra with Killing form $B$.
Let $\frr\subset\frg$ be a reductive Lie subalgebra such that
$B|_{\frr\times \frr}$  is non-degenerate. Let
$\frg=\frr\oplus\frs$ be the orthogonal decomposition with respect
to $B$. Then the restriction
$B|_\frs$ is also non-degenerate. Denote by $C(\frs)$ the Clifford
algebra of $\frs$ with
\begin{equation*}
uu'+u'u=-2B(u, u')
\end{equation*}
for all $u, u'\in\frs$. The above choice of sign is the same as in [HP2], but
different from the definition in [Ko1], as well as in [HPR].  The two different
choices of signs have no essential difference
since  the two bilinear forms are equivalent over $\mathbb{C}$. Now
fix an orthonormal basis $Z_1, \ldots, Z_m$ of $\frs$. Kostant
\cite{Ko1} defines the cubic Dirac operator $D$ by
\begin{equation*}
D=\sum_{i=1}^m{Z_i\otimes Z_i+1\otimes v}\in U(\frg)\otimes C(\frs).
\end{equation*}
Here $v\in C(\frs)$ is the image of the fundamental 3-form
$w\in\bigwedge^3(\frs^*)$,
\begin{equation*}
w(X,Y,Z)=\frac{1}{2}B(X,[Y,Z]),
\end{equation*}
under the Chevalley map $\bigwedge(\frs^*)\rightarrow C(\frs)$ and
the identification of $\frs^*$ with $\frs$ by the Killing form $B$.
Explicitly,
\begin{equation*}
v=\frac{1}{2}\sum_{1\leq i<j<k \leq m}B([Z_i, Z_j],Z_k)Z_iZ_jZ_k.
\end{equation*}

The cubic Dirac operator has a good square in analogue with
the Dirac operator associated with the symmetric pair $(\frg,\frk)$ in Section 2.  We
have  a similar Lie algebra map
\begin{equation*}
\alpha:\frr\rightarrow C(\frs)
\end{equation*}
which is defined by the adjoint map $\ad:\frr\rightarrow\frso(\frs)$
composed with the embedding of $\frso(\frs)$ into $C(\frs)$ using
the identification $\frso(\frs)\simeq\bigwedge^2\frs$. The explicit
formula for $\alpha$ is (see \S 2.3.3 [HP2])
\begin{equation}\label{mapalpha}\alpha(X)=-\frac{1}{4}\sum_{j}[X, Z_j]Z_j, \quad\
X\in\frr.
\end{equation}
Using $\alpha$ we can embed the Lie algebra $\frr$ diagonally into
$U(\frg)\otimes C(\frs)$, by
\begin{equation*}\label{Delta map}
X\mapsto X_\Delta=X\otimes1+1\otimes\alpha(X).
\end{equation*}
This embedding extends to $U(\frr)$. We denote the image of $\frr$
by $\frr_\Delta$, and then the image of $U(\frr)$ is the enveloping
algebra $U(\frr_\Delta)$ of $\frr_\Delta$. Let $\Omega_\frg$ (resp.
$\Omega_\frr$) be the Casimir elements for $\frg$ (resp. $\frr$).
The image of $\Omega_\frr$ under $\Delta$ is denoted by
$\Omega_{\frr_\Delta}$.

Let $\frh_\frr$ be a Cartan subalgebra of $\frr$ which is contained
in $\frh$. It follows from Kostant's calculation (\cite{Ko1}, Theorem 2.16) that
\begin{equation}\label{square}
D^2=-\Omega_\frg\otimes1+\Omega_{\frr_\Delta}-(\|\rho\|^2 - \|\rho_{\frr}\|^2)1\otimes1,
\end{equation}
where $\rho_\frr$ denote the half sum of positive roots for $(\frr,
\frh_\frr)$.  We also note the sign difference with Kostant's formula
due to our choice of bilinear form for the definition of the Clifford algebra $C(\frs)$.

We denote by $W$ the Weyl group associated to the root system
$\Delta(\frg,\frh)$ and $W_\frr$ the Weyl group associated to the root system
$\Delta(\frr,\frh_\frr)$. The following theorem due to Kostant is an
extension of Vogan's conjecture on the symmetric pair case which
is proved in [HP1].
(See \cite{Ko3} Theorems 4.1 and 4.2 or \cite{HP2} Theorem 4.1.4).

\begin{theorem}\label{Vogan's conjecture1}
There is an algebra homomorphism $\zeta:
Z(\frg)\rightarrow Z(\frr)\cong Z(\frr_{\Delta})$ such that for any $z\in Z(\frg)$ one has
\begin{equation*}\label{Vogan's equation}
z\otimes 1-\zeta(z)=Da+aD \text{\ for some\ }a\in U(\frg)\otimes C(\frs).
\end{equation*}
Moreover, $\zeta$ is determined
by the following commutative diagram:
\begin{equation*}\label{Vogan's diagram}
\CD
  Z(\mathfrak{g}) @> \zeta >> Z(\mathfrak{r}) \\
  @V \eta VV @V \eta_{\mathfrak{r}} VV  \\
  P(\mathfrak{h}^*)^{W} @>\mathrm{Res}>>
  P(\mathfrak{h}_\frr^*)^{W_{\mathfrak{r}}}.
\endCD
\end{equation*}
Here the vertical maps $\eta$ and $\eta_\frr$ are Harish-Chandra
isomorphisms.
\end{theorem}

\begin{definition} Let $S$ be a spin module of $C(\frs)$.
Consider the action of $D$ on $V\otimes S$
\begin{equation}\label{Dirac map}
D:V\otimes S\rightarrow V\otimes S
\end{equation}
with $\frg$ acting on $V$ and $C(\frs)$ on $S$.  The {\it Dirac
cohomology} of $V$ is defined to be the $\frr$-module
\begin{equation*}
H_D(V):=\Ker D/\Ker D\cap \Im D.
\end{equation*}
\end{definition}

The following theorem is a consequence of the above theorem.
\begin{theorem}[\cite{Ko3},\cite{HP2}]
Let $V$ be a $\frg$-module with
$Z(\frg)$ infinitesimal character $\chi_\Lambda$. Suppose that an
$\frr$-module $N$ is contained in the Dirac cohomology
$H_D(V)$ and has $Z(\frr)$ infinitesimal
character $\chi_\lambda$ . Then $\lambda=w\Lambda$ for some $w\in W$.
\end{theorem}

Suppose that $V_\lambda$ is a finite-dimensional representation
with highest weight $\lambda\in \frh^*$.  Kostant [K2] calculated
the Dirac cohomology of $V_\lambda$ with respect to any equal rank
quadratic subalgebra $\frr$ of $\frg$.  Assume that $\frh\subset\frr\subset\frg$ is
the Cartan subalgebra for both $\frr$ and $\frg$.
Define $W(\frg,\frh)^1$ to be the subset of the Weyl group $W(\frg,\frh)$ by
$$
W(\frg,\frh)^1=\{w\in W(\frg,\frh)\ | w(\rho) \text{ is }\Delta^+(\frr,\frh)-\text{dominant}
\}.
$$
This is the same as the subset of elements $w\in W(\frg,\frh)$ that map the positive Weyl $\frg$-chamber
into the positive $\frr$-chamber.  There is a bijection
$W(\frr,\frh)\times W(\frg,\frh)^1\rightarrow W(\frg,\frh)$ given by $(w,\tau)\mapsto w\tau$.
Kostant proved [K2] that
$$
H_D(V_\lambda)=\bigoplus_{w\in W(\frg,\frh)^1}E_{w(\lambda+\rho)-\rho_\frr}.
$$

The above result of Kostant on Dirac cohomology of
finite-dimensional modules has been extended to unequal rank case by Mehdi and Zierau [MZ].
We now show how to calculate the Dirac cohomology of a simple highest
weight module of possibly infinite dimension.
We need to recall the definition
and some of the basic properties of the category $\caO^\frq$
associated with an arbitrary parabolic subalgebra $\frq$ of $\frg$.

Recall that if  $\frg$ is  a complex semisimple Lie algebra with  a
Cartan subalgebra of $\frh$, we denote by $\Phi=\Delta(\frg,\frh)\subseteq\frh^*$ the
root system of ($\frg$, $\frh$). For $\alpha\in\Phi$, let
$\frg_\alpha$ be the root subspace of $\frg$ corresponding to
$\alpha$. We fix a choice of the set of positive roots $\Phi^+$ and let
$\Delta$ be the corresponding subset of simple roots in $\Phi^+$.
Note that each subset $I\subset\Delta$ generates a root system
$\Phi_I\subset\Phi$, with positive roots $\Phi_I^+=\Phi_I\cap
\Phi^+$.

The parabolic subalgebras of $\frg$ up to conjugation
are in one to one correspondence with the subsets in $\Delta$.
We let
$$\frl_I=\frh\oplus \sum_{\alpha\in\Phi_I}{\frg_\alpha}$$
be the Levi subalgebra and
$$\fru_I=\sum_{\alpha\in\Phi^+\backslash\Phi_I^+}{\frg_\alpha}, \
\bar{\fru}_I=\sum_{\alpha\in\Phi^+\backslash\Phi_I^+}{\frg_{-\alpha}}$$ be
the nilpotent radical and its dual space with respect to the Killing
form $B$. Then $\frq_I=\frl_I\oplus\fru_I$ is the standard parabolic
subalgebra associated with $I$. We set
$$\rho=\rho(\frg)=\frac{1}{2}\sum_{\alpha\in\Phi^+}{\alpha}, \
\rho(\frl_I)=\frac{1}{2}\sum_{\alpha\in\Phi_I^+}{\alpha}, \text{\
and \ }
\rho(\fru_I)=\frac{1}{2}\sum_{\alpha\in\Phi^+\backslash\Phi_I^+}{\alpha}.$$
Then we have $\rho(\bar{\fru}_I)=-\rho(\fru_I)$. We note that once
$I$ is fixed there is little use for other subsets of $\Delta$.
We will omit the subscript if a subalgebra  is
clearly associated with $I$.

\begin{definition}
The category $\mathcal{O}^\frq$ is defined to be the full
subcategory of $U(\frg)$-modules $M$ that satisfy the following
conditions:
\begin{itemize}
\item[(i)] $M$ is a finitely generated $U(\frg)$-module;

\item[(ii)]  $M$ is a direct sum of finite dimensional simple
$U(\frl)$-modules;

\item[(iii)] $M$ is locally finite as a $U(\frq)$-module.
\end{itemize}
\end{definition}

We adopt notation of \cite{Hu}. Let $\Lambda_I^+$ be the set of
$\Phi_I^+$-dominant integral weights in $\frh^*$, namely,
\begin{equation*}
\Lambda_I^+ :=\{\lambda\in\frh^*\ |\ \langle\lambda,
\alpha^\vee\rangle\in\mathbb{Z}^{\geq0}\ \mbox{for all}\
\alpha\in\Phi_I^+\}.
\end{equation*}
Here $\langle, \rangle$ is the bilinear form on $\frh^*$ (induced
from the Killing form $B$) and $\alpha^\vee=2\alpha/\langle\alpha,
\alpha\rangle$.

Let $F(\lambda)$ be the finite-dimensional simple $\frl$-module with
highest weight $\lambda$. Then $\lambda\in\Lambda_I^+$. We consider
$F(\lambda)$ as a $\frq$-module by letting $\fru$ act trivially on
it. Then the {\it parabolic Verma module} with highest weight
$\lambda$ is the induced module
\begin{equation*}
M_I(\lambda)=U(\frg)\otimes_{U(\frq)}F(\lambda).
\end{equation*}
The module $M_I(\lambda)$ is a quotient of the ordinary Verma module
$M(\lambda)$. Using Theorem 1.2 in \cite{Hu}, we can write
unambiguously $L(\lambda)$ for the unique simple quotient of
$M_I(\lambda)$ and $M(\lambda)$. Furthermore, since every nonzero
module in $\mathcal{O}^\frq$ has at least one nonzero vector of
maximal weight, Proposition 9.3 in \cite{Hu} implies that every
simple module in $\mathcal{O}^\frq$ is isomorphic to
$L(\lambda)$ for some $\lambda\in\Lambda_I^+$ and is therefore
determined uniquely up to isomorphism by its highest weight.

Recall that $M_I(\lambda)$
and all its subquotients including $L(\lambda)$ have the same
infinitesimal character
$$\chi_\lambda\colon Z(\frg) \rightarrow \bbC.$$  Here  $\chi_\lambda$ is an
algebra homomorphism  such that
$z\cdot v=\chi_\lambda(z)v$ for all $z\in Z(\frg)$ and all $v\in
M(\lambda)$.  We note that the standard parameter for the infinitesimal
character $\chi_\lambda$ is the Weyl group orbit of $\lambda+\rho\in \frh^*$ due to the $\rho$-shift in
the Harish-Chandra isomorphism $Z(\frg)\rightarrow S(\frh)^W$.

It follows from  Corollary 1.2 in \cite{Hu} that every nonzero
module $M\in\mathcal{O}^\frq$ has a finite filtration with nonzero
quotients each of which is a highest weight module in
$\mathcal{O}^\frq$. Then the action of $Z(\frg)$ on $M$ is finite.
We set
\begin{equation*}
M^\chi:=\{v\in M\ |\ (z-\chi(z))^nv=0 ~~\mbox{for some}~~ n>0
~~\mbox{depending on}~~ z\}.
\end{equation*}
Then $z-\chi(z)$ acts locally nilpotently on $M^\chi$ for all $z\in
Z(\frg)$ and $M^\chi$ is a $U(\frg)$-submodule of $M$.
 Let $\mathcal{O}^\frq_\chi$
denote the full subcategory of $\mathcal{O}^\frq$ whose objects are
the modules $M$ for which $M=M^\chi$. By the above discussion we
have the following direct sum decomposition:
\begin{equation*}
\mathcal{O}^\frq=\bigoplus_{\chi}
\mathcal{O}^\frq_\chi,
\end{equation*}
where $\chi$ is of the form
$\chi=\chi_\lambda$ for some $\lambda\in\frh^*$.

Let $W$ be the Weyl group associated to the root system $\Phi$. We
define the dot action of $W$ on $\frh^*$ by
$w\cdot\lambda=w(\lambda+\rho)-\rho$ for $\lambda\in\frh^*$. Then
$\chi_\lambda=\chi_\mu$ if and only if $\lambda\in W\cdot\mu$ by the
Harish-Chandra isomorphism $Z(\frg)\rightarrow S(\frh)^W$. An
element $\lambda\in \frh^*$ is called {\it regular} if the isotropy
group of $\lambda$ in $W$ is trivial. In other words, $\lambda$ is
regular if $\langle\lambda+\rho, \alpha^\vee\rangle\neq 0$ for all
$\alpha\in\Phi$.  A non-regular element in $\frh^*$ will be called
{\it singular}.

Denote by $\Gamma$ the set of all $\mathbb{Z}^{\geq0}$-linear
combinations of simple roots in $\Delta$. Let $\mathcal {X}$ be the
additive group of functions $f: \frh^*\rightarrow \mathbb{Z}$ whose
support lies in a finite union of sets of the form $\lambda-\Gamma$
for $\lambda\in\frh^*$. Define the convolution product on $\mathcal
{X}$ by
\begin{equation*}
(f*g)(\lambda):=\sum_{\mu+\nu=\lambda}f(\mu)g(\nu).
\end{equation*}
We regard $e(\lambda)$ as a function in $\mathcal{X}$ which takes
value $1$ at $\lambda$ and value $0$ at $\mu\neq\lambda$. Then
$e(\lambda)*e(\mu)=e(\lambda+\mu)$. It is clear that $\mathcal{X}$
is a commutative ring under convolution, with $e(0)$ as its
multiplicative identity. Let
\begin{equation*}
M_\lambda:=\{v\in M\ |\ h\cdot v=\lambda(h)v\ \mbox{for all}\
h\in\frh\}.
\end{equation*}
We say that a weight module (semisimple $\frh$-module) $M$ {\it has a
character} if
\begin{equation}\label{character}
\ch M:=\sum_{\lambda\in\frh^*}\dim M_\lambda\ e(\lambda)
\end{equation}
is contained in $\mathcal{X}$. In this case, $\ch M$ is called the
{\it formal character} of $M$. Notice that all the modules in
$\mathcal{O}^\frq$ have characters, as do all the finite dimensional
semisimple $\frh$-modules. In particular, if $M$ has a character and $\dim
L<\infty$, then $M\otimes L$ has a character
$$\ch(M\otimes L)=\ch M*\ch L.$$
In addition, for semisimple $\frh$-modules which have characters, their
direct sums, submodules and quotients also have characters.

As a consequence of established Vogan's conjecture for
the cubic Dirac operators we have the following proposition (See also [DH1] Theorem 4.3).

\begin{prop} Suppose that $V$ is in $\caO^\frq_{\chi_\mu}$.  Then the Dirac cohomology
$H_D(V)$ is a completely reducible finite-dimensional $\frl$-module.
Moreover, if the finite-dimensional $\frl$-module $F(\lambda)$ is contained in
$H_D(V)$, then
$\lambda+\rho_\frl=w(\mu+\rho)$ for some $w\in W$.
\end{prop}

It is shown in [HX] that determination of
$H_D(L(\lambda))$ is equivalent to determination of
$\ch L(\lambda)$ in terms of  $\ch M_I(\mu)$, which is
solved by the Kazhdan-Lusztig algorithm.  Namely, if
$$\ch L(\lambda)=\sum (-1)^{\epsilon(\lambda,\mu)}m(\lambda,\mu)\ch
M_I(\mu),$$
then we have
$$H_D(L(\lambda))=\bigoplus
m(\lambda,\mu)F(\mu)\otimes\mathbb{C}_{\rho(\fru)}.$$
By using the known results on Kazhdan-Lusztig polynomials we can
calculate explicitly the Dirac cohomology of all irreducible highest weight modules.
We recall the main result from [HX] here.
Recall that $W=W(\frg,\frh)$ is the Weyl group
associated to the root system $\Phi$.
We define
\begin{equation*}
\Phi_{[\lambda]} :=\{\alpha\in\Phi\ |\ \langle\lambda,
\alpha^\vee\rangle\in \mathbb{Z}\}.
\end{equation*}
Then it is the root system of integral roots associated to $\lambda$.
We also set
\begin{equation*}
W_{[\lambda]} :=\{w\in W\ |\ w\lambda-\lambda\in\Lambda_r \},
\end{equation*}
where $\Lambda_r$ is the $\mathbb{Z}$-span of $\Phi$. Then $W_I$ is
contained in $W_{[\lambda]}$. We also define
\begin{equation*}
W^I=\{w\in W_{[\lambda]}\ |\ w<s_\alpha w\ \mbox{for all $\alpha\in
I$}\},
\end{equation*}
where the ordering on $W$ is given by the Bruhat ordering. Denote by
$\Delta_{[\lambda]}$ the simple system corresponding to the positive
system $\Phi_{[\lambda]}\cap\Phi^+$ in $\Phi_{[\lambda]}$.
Let $\mu$ be the unique
anti-dominant weight in $W_{[\lambda]}\cdot\lambda$. The set
of singular simple roots in $\Delta_{[\lambda]}$ is defined by
\begin{equation*}
J=\{\alpha\in\Delta_{[\lambda]}\ |\ \langle\mu+\rho,
\alpha^\vee\rangle=0\}.
\end{equation*}
Then $W_J=\{w\in W\ |\ w(\mu+\rho)=\mu+\rho\}\subset W_{[\lambda]}$
is the isotropy group of $\mu$. Put
\begin{equation*}
{}^JW^I=\{w\in W^I\ |\ w<ws_\alpha\ \mbox{and}\ ws_\alpha\in
W^I~~\mbox{for all}~~\alpha\in J\}.
\end{equation*}

Following Boe and Hunziker [BH] we define
\begin{equation*}
{}^JP_{x,w}^I(q)=\sum_{i\geq
0}{q^{\frac{l(w)-l(x)-i}{2}}}\dim\Ext_{\mathcal{O}^\frp}^i(M_x,
L_w), \text{\ for all\ }x, w\in {}^JW^I.
\end{equation*}
It is shown to be a polynomial and is called
the relative {\it Kazhdan-Lusztig-Vogan polynomial}.

\begin{theorem} [Theorem 6.16 \cite{HX}]  If $L(\lambda)$ is
the simple highest weight module in $\mathcal{O}^\frp_\mu$ of weight
$\lambda=w_Iw\cdot\mu$ with $w_I$ the longest element in $W_I$, then
one has an $\frl$-module decomposition
\begin{equation*}
H_D(L(\lambda))\simeq\bigoplus_{x\in{}^{J}W^I}{}^{J}P_{x,w}^{I}(1)F({w_{I}x\cdot\mu+\rho(\fru)}).
\end{equation*}
\end{theorem}

\begin{remark}
Applying the action of Chevalley automorphism (see section 6) on Dirac cohomology,
we can also determine Dirac cohomology of simple lowest
weight modules.
\end{remark}

\section{Dirac cohomology and $\fru$-cohomology}
%
In this section we review the results on Dirac cohomology
and $\frp^+$-cohomology of unitary representations for Hermitian symmetric case.
Then we discuss the simple highest weight modules in $\caO^\frq$.
We use quite different techniques for these two cases.

Suppose that $G$ is simple and Hermitian symmetric, with
maximal compact subgroup $K$.  In this case the $\frk$-module $\frg$
decomposes as
$\frg=\frk\oplus\frp=\frk\oplus\frp^+\oplus\frp^-$.
We can choose the basis $b_i$ of $\frp$ in the following special way. Let
$\Delta_n^+=\{\beta_1,\dots,\beta_m\}$. For each $\beta_i$ we
choose a root vector $e_i\in\frp^+$. Let $f_i\in\frp^-$ be the root vector
for the
root $-\beta_i$ such that $B(e_i,f_i)=1$. Then for the basis
$b_i$ of $\frp$ we choose $e_1,\dots e_m;\, f_1,\dots f_m$. The
dual basis is then $f_1,\dots f_m;\, e_1,\dots e_m$. Thus the Dirac operator is
\begin{equation*}
D=\sum_{i=1}^m e_i\otimes f_i + f_i\otimes e_i.
\end{equation*}

We also note that in this case $G$ is of equal rank and in particular
$\frp$ is even-dimensional.  Therefore,
there is a unique irreducible $C(\frp)$-module, the
spin module $S$, which we choose to construct as $S=\bigwedge\frp^+$.
It is also a module for the double cover $\Kt$ of $K$.
Let $X$ be a  $(\frg,K)$-module. Since $\frp^+\cong(\frp^{-})^{*}$, we have
\eq
\label{spinwedge}
X\otimes S\cong X\otimes
{\textstyle\bigwedge}\frp^+\cong\Hom({\textstyle\bigwedge}\frp^-,X)
\eeq
as vector spaces.
Note that the underlying vector space ${\textstyle\bigwedge}\frp^+$ of the
spin module $S$ carries the adjoint action of
$\frk$, but the relevant $\frk$-action on $S$
is the spin action defined using the map (\ref{mapalpha}). The spin
action is equal to the
adjoint action shifted by the character $-\rho_n$ of $\frk$ (see
\cite[Proposition~3.6]{Ko2}).
So as a $\frk$-module, $X\otimes S$ differs from $X\otimes
{\textstyle\bigwedge}\frp^+$ and
$\Hom({\textstyle\bigwedge}\frp^-,X)$ by a twist by the $1$-dimensional
$\frk$-module $\bbC_{-\rho_n}$.

Let $C=\sum_{i=1}^mf_i\otimes e_i$ and $C^-=\sum_{i=1}^me_i\otimes f_i$;
so $D=C+C^-$. Then, under the identifications
(\ref{spinwedge}), $C$ acts on $X\otimes S$ as the $\frp^-$-cohomology
differential, while $C^-$ acts by $2$ times the $\frp^+$-homology
differential (see \cite[Proposition~9.1.6]{HP2} or \cite{HPR}).
Furthermore, $C$ and $C^-$ are adjoints of each other with respect to the
Hermitian inner product on $X\otimes S$ mentioned above (see
\cite[Lemma~9.3.1]{HP2} or \cite{HPR}).
It was proved that Dirac cohomology is isomorphic to
$\frp^-$-cohomology up to a one-dimensional character
by using a version of Hodge decomposition.

\begin{theorem}[\cite{HPR}, Theorem 7.11]\label{p^+cohomology}
Let $X$ be a unitary $(\frg,K)$-module. Then
\[
H_D(X)\cong H^*(\frp^-,X)\otimes \bbC_{-\rho(\frp^+)}\cong H_*(\frp^+,X)\otimes \bbC_{-\rho(\frp^+)}
\]
as $\frk$-modules. Moreover, the above isomorphisms hold for a parabolic subalgebra
$\frq=\frl+\fru$ as long as $\frl\subseteq \frk$ and $\fru\supseteq \frp^+$, that is
\[
H_D(X)\cong H^*(\fru^-,X)\otimes \bbC_{-\rho(\fru)}\cong H_*(\fru,X)\otimes \bbC_{-\rho(\fru)}
\]
as $\frl$-modules.
\end{theorem}

Note that we may use $\bigwedge\frp^-$ instead of $\bigwedge\frp^+$ to
construct the spin module $S$.  Then we have
\begin{equation}\label{p+cohomology}
H_D(X)\cong H^*(\frp^+,X)\otimes \bbC_{\rho(\frp^+)} \cong H_*(\frp^-,X)\otimes \bbC_{\rho(\frp^+)}.
\end{equation}
Namely, the Dirac operator is independent of the choice
of positive roots. Thus, we also have
\[
H^*(\frp^+,X)\otimes \bbC_{\rho(\frp^+)}\cong H^*(\frp^-,X)\otimes \bbC_{-\rho(\frp^+)},
\]
and
\[
H_*(\frp^+,X)\otimes \bbC_{-\rho(\frp^+)}\cong H_*(\frp^-,X)\otimes \bbC_{\rho(\frp^+)}.
\]
It also follows that we know the Dirac cohomology of all irreducible unitary
highest weight modules explicitly from Enright's calculation of $\frp^+$-cohomology [E].

Now suppose $\frq=\frl+\fru$ is a parabolic subalgebra of $\frg$ as in section 4.
We recall the result from [HX] on relation between Dirac cohomology with respect
to $D(\frg,\frl)$ and $\fru$-cohomology.
We note that the spin action $\alpha(\frl)$ on S makes it a finite-dimensional
$\frl$-module. If $V\in\mathcal{O}^\frq$, then $V\otimes
S$ is a direct sum of finite-dimensional simple $\frl$-modules.
Hence, any submodule, quotient or subquotient of $V\otimes S$ is
also a direct sum of finite-dimensional simple $\frl$-modules.

Then the Casimir element $\Omega_\frg$ acts semisimply on $V$.
We have shown that $H_D(V)$ is isomorphic to $\fru$-homology up to a character in [HX].
We recall here the main steps of the proof of this isomorphism ([HX] Theorem 5.12).
The $\fru$-homology is $\bbZ_2$-graded as follows:
$$H_+(\fru, V)=\bigoplus_{i=0} H_{2i}(\fru, V)\ \mbox{and}\ H_-(\fru,
V)=\bigoplus_{i=0} H_{2i+1}(\fru, V).$$
Then there are injective $\frl$-module homomorphisms ([HX] Proposition 4.8):
$$\label{equal for simple}
H^{\pm}_D(V)\rightarrow H_\pm(\fru, V)\otimes
\mathbb{C}_{\rho(\bar\fru)}.$$
Note that we also have ([HX] Proposition 5.2)
$$  \ch H_D^+(V)-\ch H_D^-(V)=(\ch H_+(\fru, V)-\ch H_-(\fru,
V))*\ch\mathbb{C}_{\rho(\bar{\fru})}.$$
Then the properties of KLV polynomials (see Proposition 6.14 [HX]) imply the following parity
condition:  $H_+(\fru, V)$
and $H_-(\fru, V)$ have no common $\frl$-submodules, namely,
$$
 \Hom_\frl(H_+(\fru, V), H_-(\fru, V))=0.
$$
This is the key lemma in [HX] (Lemma 5.11). It follows that the embeddings
$$H^{\pm}_D(V)\rightarrow H_\pm(\fru, V)\otimes
\mathbb{C}_{\rho(\bar\fru)}$$ are isomorphisms.

\begin{theorem} [Theorem 5.12, Corollary 5.13 \cite{HX}]
Let $V$ be a simple highest weight module in
$\mathcal{O}^\frp$. Then we have the following $\frl$-module
isomorphisms:
$$
H_D(V)\simeq H_*(\fru, V)\otimes \mathbb{C}_{-\rho(\fru)}\simeq H^*(\bar\fru, V)\otimes \mathbb{C}_{-\rho(\fru)}
.$$
\end{theorem}
As mentioned, the special case when $\frq$ contains $\frp^+$ in the hermitian symmetric setting
the above theorem is proved for unitary Harish-Chandra modules in
[HPR].  The argument in [HPR] depends on the existence of the positive definite Hermitian form on
these modules .
This argument cannot be extended to any simple highest weight modules.

For a Harish-Chandra module $X$ for $G$ of hermitian symmetric type,
we have similar injective homomorphisms
of
$$H^\pm_D(X) \rightarrow H_\pm(\frp^+,X)\otimes \bbC_{-\rho(\frp^+)}.$$.

\begin{conj} Suppose that $X$ is a simple Harish-Chandra module.
Then
$$ \Hom_{\Kt}(H_+(\frp^+, X), H_-(\frp^+, X))=0.$$
In particular, it implies the above injective homomorphisms are actually isomorphisms.
In other words, we conjecture that Theorem \ref{p^+cohomology} holds
for any simple $(\frg,K)$-module $X$.
\end{conj}

We note that Dirac cohomology for unitary lowest weight modules
is an important ingredient for generalization of classical branching
rules [HPZ] [H2].

\begin{remark} It is obvious that all the proofs for the theorems in this section on highest weight modules
can be done for lowest weight modules.
\end{remark}

We review a result about
action of an automorphism on Dirac cohomology [HPZ].   The above remark is
a consequence by taking the Chevalley automorphism.
Let $\tau$ be an automorphism of $G$ preserving $K$. Then $\tau\big|_K$ is
an automorphism of $K$. Also, $\tau$ induces automorphisms of $\frg_0$ and
$\frg$, denoted again by $\tau$, and $\tau$ preserves the Cartan
decomposition
$\frg=\frk\oplus\frp$. Finally, $\tau\big|_\frp$ extends to an
automorphism of the Clifford algebra $C(\frp)$, denoted again by $\tau$.
Let $X$ be a $(\frg,K)$-module. If we set $X^\tau=X$, then
$(\pi\circ\tau,X^\tau)$ is also a
$(\frg,K)$-module. Similarly, for any $K$-module $(\varphi,E)$, if we set
$E^\tau=E$, then $(\varphi\circ\tau,E^\tau)$
is also a $K$-module. The same is true if we replace $K$ by $\Kt$. The
following property of Dirac cohomology
was proved for any unitary $(\frg,K)$-module in [HPZ] (Prop. 5.1 of [HPZ]).
The same proof extends straightforwardly to any $(\frg,K)$-module (Prop. 3.12 [H2]).

\begin{prop}
\label{twist dirac}
Let $X$ be a $(\frg,K)$-module. Then
\[
H_D(X^\tau)\cong (H_D(X))^\tau.
\]
\end{prop}

%
\section{Calculation of Dirac cohomology in stages}

In this section we review another technique of calculating Dirac cohomology, namely calculation in stages.
This technique is needed to study Dirac cohomology of elliptic presentations.
Recall that $\frg_0$ is the Lie algebra of $G$ with a compact subalgebra $\frk_0$, the Lie algebra of $K$.
We assume $\frk_0$ is of equal rank with $\frg_0$.  Then a Cartan subalgebra $\frt_0$
of $\frk_0$
is also a Cartan subalgebra of $\frg_0$.  We drop the subscripts for their complexification.
The bilinear form $B$ on $\frg$
is non-degenerate and the restriction $B|_\frt$ remains to be non-degenerate.
Then we have an orthogonal decomposition
$\frg=\frt\oplus\frs$ with $\frs$ the orthogonal complement of $\frt$ with respect to $B$. It follows that
$B|_{\frs}$ is also non-degenerate. The cubic Dirac operator $D(\frg,\frt)$ due to
Kostant \cite{Ko2} is in $U(\frg)\otimes C(\frs)$. Let $\{Y_i\}_{i=1}^n$ be an orthonormal basis of $\frs$.
 Then we can write (see \cite[4.1.1]{HP2})
$$D(\frg,\frt)=\sum_{i=1}^nY_i\otimes Y_i +{1\over 2}\sum_{i<j<k}B([Y_i,Y_j],Y_k)\otimes Y_iY_jY_k
.$$

Let $\frg=\frk\oplus\frp$ be the complexification of
the Cartan decomposition and let $\frs_1$ be the orthogonal complement of $\frt$ in
$\frk$.  Then $\frs=\frs_1\oplus \frp$.
As in \cite[\S 9.3]{HP2} we write the Dirac operator
$D(\frg,\frt)$ in terms of $D(\frg,\frk)$ and $D(\frk,\frt)$
by using an orthonormal basis for $\frs$ formed by
orthonormal bases $Z_i$ for $\frp$ and $Z'_j$ for $\frs_1$.

Identifying
$ U(\frg)\otimes C(\frs)$ with $U(\frg)\otimes
C(\frp)\bar\otimes C(\frs_1)$,
where $\bar\otimes$ denotes the $\bbZ_2$-graded tensor product, we
can write
\begin{equation}
\label{eqn.3.4}
\begin{array}{rl}
D(\frg,\frt) = & \sum_i Z_i \otimes Z_i \otimes 1 + \sum_j Z_j' \otimes 1\otimes Z_j' \\
& +\frac{1}{2} \sum_{i<j}\sum_k  B([Z_i,Z_j],Z_k')\otimes Z_i Z_j \otimes Z_k' \\
& +\frac{1}{2} \sum_{i<j<k} B([Z_i',Z_j'],Z_k')\otimes 1 \otimes Z_i' Z_j' Z_k'.
\end{array}
\end{equation}
Regarding $U(\frg)\otimes C(\frk)$ as the subalgebra
$U(\frg)\otimes C(\frk)\otimes 1$ of $U(\frg)\otimes
C(\frk)\bar\otimes C(\frs_1)$,
the first summand in (\ref{eqn.3.4})  gives
$D(\frg,\frk)$ and the remaining three summands in (\ref{eqn.3.4})
come from the cubic Dirac operator
corresponding to $\frt\subset\frk$. However, this is an element of
the algebra $U(\frk)\otimes C(\frs_1)$, and this algebra has
to be embedded into $U(\frg)\otimes C(\frp)\bar\otimes
C(\frs_1)$ diagonally, by
$$
\Delta: U(\frk)\otimes C(\frs_1) \cong
U(\frk_\Delta)\bar\otimes C(\frs_1) \subset U(\frg)\otimes
C(\frp)\bar\otimes C(\frs_1).
$$
Here $U(\frk_\Delta)$ is embedded into $U(\frg)\otimes C(\frp)$ by a
diagonal embedding while the factor $C(\frs_1)$ remains
unchanged.  We denote the image $\Delta(D(\frk,\frt))$ by $D_\Delta(\frk,\frt)$.

\begin{theorem} [Theorem 3.2 \cite{HPR}, Theorem 9.4.1 \cite{HP2}]
With notation as above,
$D(\frg,\frt)$ decomposes as
$D(\frg,\frk)+D_\Delta(\frk,\frt)$. Moreover, the summands
$D(\frg,\frk)$ and $D_\Delta(\frk,\frt)$ anti-commute.
\end{theorem}

The above decomposition holds in a slightly more general setting
as it is stated and proved in
Theorem 3.2 of \cite{HPR}
and in even more general setting in Theorem 9.4.1 \cite{HP2}.
The anti-commuting property given here can be applied to
calculate Dirac cohomology in stages.
For convenience,  we define the cohomology of any linear operator
$A$ on a vector space $V$ to be the vector space
$$H(A)=\ker A/(\im A\cap\ker A).$$
We also denote by $H(A;V)$ for the cohomology when we
emphasize the space $V$.
We call the operator $A$ semisimple if $V$ is the
(algebraic) direct sum of the eigenspaces of $A$.

Let $S$ be the simple module of the Clifford algebra $C(\frs)$.
If $X$ is a $(\frg,K)$-module, then $D(\frg,\frt)$ acts on $X\otimes S$.
We denote by $H_D(\frg,\frt;X)$ the cohomology of $X\otimes S$
with respect to $D(\frg,\frt)$; analogous notation will be used for
other Dirac operators.  We note that $H_D(\frg,\frt;X)$ is in fact
the cohomology of the operator $D(\frg,\frt)$ on $X\otimes S$, namely
$$H_D(\frg,\frt;X)=H(D(\frg,\frt);X\otimes S).$$

\begin{lemma} [Lemma 5.3 \cite{HPR}]
\label{lem.9.4.2} Let $A$ and $B$ be anti-commuting linear operators
on an arbitrary vector space $V$.
Assume that $A^2$ and $B$ are semisimple. Then $H(A+B)$ is the
cohomology (i.e., the kernel) of $B$ acting on $H(A)$.
\end{lemma}

The above theorem and lemma imply the following theorem for calculating
Dirac cohomology in stages.
\begin{theorem}[Theorem 6.1 \cite{HPR}, Theorem 9.4.4 \cite{HP2}]
Let $X$ be an admissible $(\frg,K)$-module with $\Omega_\frg$ acting semisimply. Then we have
$$H_D(\frg,\frt;X)=H_D(\frk,\frt;H_D(\frg,\frk;X)).$$
Also, we can reverse the order to have
$$H_D(\frg,\frt;X)=H(D(\frg,\frk)|_{H_D(\frk,\frt;X)}).$$
\end{theorem}
The above theorem is proved in \cite{HPR} and in \cite{HP2} for a slightly more general case
when $\frt$ is any subalgebra of $\frk$.  In the following we use the theorem
to calculate the Dirac cohomology of discrete series as an example.

\begin{example} Suppose that $G$ is a connected semisimple Lie group
with finite center.  Let $X_\lambda$ be the Harish-Chandra module of a discrete series representation with
Harish-Chandra parameter $\lambda$.
Then the Dirac cohomology of $X_\lambda$ with
respect to $D(\frg,\frk)$ consists of a single $\Kt$-module
$E_\mu$, whose highest weight is $\mu=\lambda-\rho_c$. Here
$\rho_c$ is half the sum of roots of $\frt$ in $\frk$ positive on $\lambda$.
We note that the highest weight $\mu$ is obtained from the highest
weight $\lambda-\rho_c+\rho_n$ of the lowest $K$-type of $X_\lambda$ by adding
$-\rho_n$ (the lowest weight of the spin module $S$ for $C(\frp)$).
The  fact  $H_D(X_\lambda)=E_\mu$ follows from Theorem 2.5 , since $X_\lambda=A_\frb(\lambda-\rho)$
for a $\theta$-stable Borel subalgebra $\frb$.
This fact  can also be proved directly without using Theorem 2.5 as follows. By
\cite[Prop.5.4]{HP1}, the $\Kt$-type $\mu$ is clearly contained in the Dirac
cohomology. Since $X_\lambda$ has a unique lowest $K$-type, and since
$-\rho_n$ is the lowest weight of the spin module $S$ for $C(\frp)$, with multiplicity
one, it follows that any other $\Kt$-type has strictly larger
highest weight, and thus cannot contribute to the Dirac cohomology.
We now apply
Kostant's formula from section 4 to calculate the Dirac
cohomology of $E_\mu$ with respect to $D(\frk,\frt)$:
$$
H_D(\frk,\frt;E_\mu)= \ker D(\frk,\frt)=\bigoplus_{w\in
W_K}\bbC_{w(\mu+\rho_c)}.
$$  It follows from
$\mu+\rho_c=\lambda$ that
$$
H_D(\frg,\frt; X_\lambda)=\bigoplus_{w\in W_K}\bbC_{w\lambda}.
$$
\end{example}

\begin{remark}  In [CH] a modified
Dirac operator is defined as follows
\[
{\widetilde D}(\frg,\frt)
=D(\frg,\frk)+iD_\Delta(\frk,\frt).
\]
This is used for the geometric quantization and construction of
models of discrete series.
Let $X$ be a unitary $(\frg,K)$-module.
There is a hermitian form on the spin module $S$
\cite[\S2.3.9]{HP2}. Together with the $\frg_0$ invariant
hermitian form on $V$, it induces a hermitian form on $X\otimes S$. It follows
from the unitarity of $V$ and the property of the hermitian form on $S$
\cite[Prop. 2.3.10]{HP2} that
$D(\frg,\frk)$ is symmetric and $D_\Delta(\frk,\frt)$ is
skew-symmetric with respect to this form.  Then the modified
Dirac operator ${\widetilde D}(\frg,\frt)$
is symmetric.  Note that both $D(\frg,\frk)^2$ and
$-D_\Delta(\frk,\frt)^2$ are positive definite, so ${\widetilde
D}(\frg,\frt)^2=D(\frg,\frk)^2-D_\Delta(\frk,\frt)^2$ is also
positive definite.  Then ${\widetilde D}(\frg,\frt)$ is an elliptic differential
operator.  This is the very purpose of introducing this modified version of Dirac operator.
We also note that $iD_\Delta(\frk,\frt)$ and $D_\Delta(\frk,\frt)$
define the same Dirac cohomology.
\end{remark}
%
\section{Elliptic representations and Dirac Index}

Suppose that $G(F)$ is a real or p-adic group.  That is, $G$ is a connected reductive algebraic group
over a local field $F$ of characteristic $0$.  Arthur [A1] studied a subset $\Pi_{temp, ell}(G(F))$
of tempered representations of $G(F)$, namely elliptic tempered representations.
The set of tempered representations $\Pi_{temp}(G(F))$ includes the discrete series
and in general the irreducible constituents of representations induced from discrete series.
These are exactly the representations which occur
in the Plancherel formula for $G(F)$.
 In Harish-Chandra's theory, the character of an infinite dimensional representation $\pi$ is defined as a distribution
$$\Theta(\pi,f)=\tr \big{(} \int_{G(F)}f(x)\pi(x)dx\ \big{)}, \  \ \ \  f\in C^\infty_c(G(F) ),$$
which can be identified with a function on $G(F)$.  In other words,
$$\Theta(\pi,f)=\int_{G(F)}f(x)\Theta(\pi,x)dx, \ \ \ \ f\in C^\infty_c(G(F)),$$
where $\Theta(\pi,x)$ is a locally integrable function on $G(F)$ that is smooth on the
open dense subset $G_{reg}(F)$ of regular elements.
A representation $\pi$ is called elliptic if $\Theta(\pi,x)$ does not vanish on the set of
elliptic elements in $G_{reg}(F)$.

The central objects in [A1] are the normalized
characters $\Phi(\pi,\gamma)$, namely the functions defined by
$$\Phi(\pi,\gamma)=|D(\gamma)|^{1\over 2}\Theta(\pi,\gamma),\ \pi  \in \Pi_{temp, ell}(G(F)), \
\gamma\in G_{reg}(F),$$
where
$$D(\gamma)=\det(1-\Ad(\gamma))_{\frg/\frg_\gamma},$$
is the Weyl discriminant.  We will show how this normalized
character $\Phi(\pi,\gamma)$ is related to the Dirac cohomology of the
Harish-Chandra module of $\pi$
for a real group $G(\R)$.

From now on we only deal with the real group $G(\R)$ and we also assume that $G(\R)$ is connected.
Note that $G(\R)$ has elliptic elements if and only if
it is of equal rank with $K(\R)$.  We also assume
this equal rank condition.
Induced representations from proper parabolic subgroups are not elliptic.
Consider the quotient of Grothendieck group of the category of finite length
Harish-Chandra modules by the subspace generated by induced
representations.  Let us call this quotient group the elliptic Grothendieck
group.  Authur [A1] found an orthonormal basis of this elliptic
Grothendieck group in terms of elliptic tempered (possibly virtual) characters.
Those characters are the super tempered distributions defined by Harish-Chandra [HC4].

The tempered elliptic representations for
the real group $G(\R)$ are the representations
with non-zero Dirac index, which are studied in [Lab1].
Labesse shows that the tempered elliptic representations are precisely the fundamental series.
We deal with the general elliptic representations
and show that any elliptic representation
has nonzero Dirac index.

Recall that if $X$ is an
admissible $(\frg,K)$-module with $K$-type decomposition
$X=\bigoplus_{\lambda}m_\lambda E_\lambda$, then the $K$-character of $X$ is the formal series
\[
\ch X=\sum_{\lambda}m_\lambda\ch E_\lambda,
\]
where $\ch E_\lambda$ is the character of the irreducible $K$-module $E_\lambda$.
Moreover, this definition makes sense also for virtual $(\frg,K)$-modules $X$;
in that case, the integers $m_\lambda$ can be negative. In the following
we will often deal with representations of the spin double cover $\Kt$ of $K$,
and not $K$, but we will still denote the corresponding character by $\ch$.

Since $\frp$ is even-dimensional, the spin module $S$ decomposes as $S^+\oplus S^-$, with the $\frk$-submodules
$S^\pm$ being the even respectively odd part of $S\cong{\textstyle\bigwedge}\frp^+$.
Let $X=X_\pi$ be the Harish-Chandra module of an irreducible admissible
representation $\pi$ of $G(\bbR)$.  We consider the following difference
of $\Kt$-modules, the spinor index of $X$:
\[
I(X)=X\otimes S^+-X\otimes S^-.
\]
It is a virtual $\Kt$-module, an integer combination of
finitely many $\Kt$-modules.
The Dirac operator $D$ induces the action of the following $\Kt$-equivariant operators
\[
D^{\pm}:X\otimes S^{\pm}\rightarrow X\otimes S^{\mp}.
\]
Since $D^2$ acts by a scalar on each $\Kt$-type, most of $\Kt$-modules in
$X\otimes S^+$ are the same as in $X\otimes S^-$.

\begin{lemma} The spinor index is equal to the Euler characteristic
of Dirac cohomology, i.e.,
\[
I(X)=H_D^+(X)-H_D^-(X).
\]
\end{lemma}
\begin{proof}
As we mentioned above,  $X\otimes S$ is decomposed into
a direct sum of eigenspaces for $D^2$:
$$X\otimes S =\sum_\lambda (X\otimes S)_\lambda=(X\otimes S^+)_\lambda\oplus(X\otimes S^-)_\lambda.$$
It follows that
$$X\otimes S^+-X\otimes S^-=(X\otimes S^+)_0-(X\otimes S^-)_0.$$
Since $D$ is a differential on $\Ker D^2$ and the corresponding cohomology
is exactly the Dirac cohomology $H_D(X)$, the lemma follows from the
Euler-Poincar\'e principle.
\end{proof}

The spinor index $I(X)$ is also called the Dirac index of $X$,
since it is equal to the index of $D^+$, in the sense of index for
a Fredholm operator.  It is also identical to the Euler characteristic of Dirac cohomology $H_D(X)$.
We denote by $\theta(X)$ the character of $I(X)$.
In terms of characters, this reads
$$
\theta(X)=\ch I(X)=\ch X(\ch S^+-\ch S^-)=\ch H_D^+(X)-\ch H_D^-(X).
$$

If we view $\ch E_\lambda$ as functions on $K$, then  the series
$$\ch X=\sum_{\lambda}m_\lambda\ch E_\lambda$$
converges to a distribution on $K$ and it coincides with
$\Theta(X)$ on $K\cap  G_{reg}$, according to Harish-Chandra [HC1].
Then the absolute value $|\theta_\pi|$ coincides with the absolute value
$
|\Phi(\pi,\gamma)|=|D(\gamma)|^{1\over 2}|\Theta(\pi,\gamma)|
$
on regular elliptic elements.  We write it as a lemma.

\begin{lemma} \label{character-theta}  For any regular elliptic elements $\gamma$, we have
$$|\theta_\pi(\gamma)|=|\Phi(\pi,\gamma)|.$$
\end{lemma}

\begin{theorem}
Let $\pi$ be an irreducible admissible representation of $G(\R)$
with Harish-Chandra module $X_\pi$.
Then $\pi$ is elliptic if and only if the Dirac index $I(X_\pi)\neq 0$.
\end{theorem}
\begin{proof}  The theorem follows immediately from the lemma.
\end{proof}
%
\section{Orthogonality relations and supertempered distributions}
%
We keep the notation from the previous section.  We assume that $G(\R)$ is cuspidal, in the sense that the (regular) elliptic set
$G_{ell}$ is nonempty.
Let $A_G(\bbR)$ be the split
part of the center of $G(\bbR)$. The cuspidal condition on $G(\R)$ amounts to
the condition that $G(\R)$ has a maximal torus $T_{ell}(\bbR)$
that is compact modulo $A_G(\R)$.
Suppose $\Theta_\pi$ and $\Theta_{\pi'}$ are two irreducible
 characters with the same central character on $A_G(\bbR)$.  We form the elliptic inner
product by the following formula
$$(\Theta_\pi,\Theta_{\pi'})_{ell}=|W(G(\bbR),T_{ell}(\bbR))|^{-1}
\int_{T_{ell}(\bbR)/A_G(\bbR)}|D(\gamma)|\Theta_\pi(\gamma)\overline{\Theta_{\pi'}(\gamma)}d\gamma,$$
where $W(G(\bbR),T_{ell}(\bbR))$ is the Weyl group of $(G(\bbR),T_{ell}(\bbR))$, and $d\gamma$ is
the normalized Haar measure on the compact abelian group $T_{ell}(\bbR)/A_G(\bbR)$.
The inner product (bilinear over $\R)$) is extended linearly to any two characters of admissible representations.
Then we have the following theorem.

\begin{theorem} \label{elliptic pairing} We have
$$(\Theta_\pi,\Theta_{\pi'})_{ell}=(\theta_\pi,\theta_{\pi'}),$$
where the pairing on the right hand side is
the standard pairing of virtual characters on $K$ or $\Kt$ defined by
$$(\theta_\pi,\theta_{\pi'})=\int_K \theta_\pi\cdot \overline{\theta_{\pi'}}dk.$$

\end{theorem}

\begin{proof} It follows from Lemma \ref{character-theta} that
$$(\Theta_\pi,\Theta_{\pi})_{ell}=(\theta_\pi,\theta_{\pi})$$
for irreducible characters and therefore for all admissible characters,
in particular for any sum of two irreducible characters.
Then the theorem follows from a standard argument of polarization
for inner product.
\end{proof}


In [DH2], the set of equivalence classes of
irreducible tempered representations $\pi$ with nonzero Dirac cohomology is determined.
It turns out the irreducible tempered representations with nonzero Dirac cohomology are exactly
those with nonzero Dirac index.  Therefore, this set of representations coincides with the set of
irreducible tempered elliptic representations, and
it consists of  discrete series
and some of limits of discrete series,   Moreover, any elliptic tempered representation is isomorphic
to an $A_\frb(\lambda)$-module for some $\theta$-stable Borel subalgebra $\frb$
and its Dirac cohomology is a single $\Kt$-module.   As a consequence, we have the following corollary.

\begin{cor}\label{orthogonality}
Elliptic tempered characters satisfy orthogonality relations.  That is
for any two tempered irreducible elliptic representations $\pi$ and $\pi'$,
$$(\Theta_\pi,\Theta_{\pi'})_{ell}=\pm 1 \text{\ or } 0.$$
\end{cor}

It is also clear that the characters of discrete series form an orthonormal set on the (regular) elliptic
set $G_{ell}(\bbR)$ in $G(\bbR)$.   Harish-Chandra defined the space of supertempered distributions in
 [HC4] (the last paper of his Collected Papers Volume IV).   If $D$ is a distribution on $G$, we denote
 by $D_e$ the restriction of $D$ on $G_{ell}$ ($D_e=0$ by convention when $G_{ell}$ is empty).

\begin{theorem}
[Theorem 5 \cite{HC3}] Let $\Theta$ be a $Z(\frg)$-finite tempered distribution.
Suppose that $\Theta$ is supertempered. Then $\Theta_e=0$ implies that $\Theta=0$.
\end{theorem}

\begin{theorem}
[Theorem 9 \cite{HC3}] \label{super tempered} For $\mu\in \widehat{T(\R)}$,
there is a unique supertempered distribution $\Theta_\mu$, such that
$${}'\Delta(\gamma)\Theta_\mu(\gamma)=\sum_{w\in W_K}{\epsilon(w)}e^{w\mu},$$
where ${}'\Delta$ is the Weyl denominator (see Section 27 [HC2]).
\end{theorem}

\begin{theorem}
[Theorem 14 \cite{HC3}] If $\pi_1,\pi_2$ are irreducible tempered elliptic representations,  then
either $(\Theta_{\pi_1},\Theta_{\pi_2})_{ell}=0$ or $\Phi_{\pi_1}=\pm \Phi_{\pi_2}$.
\end{theorem}

As mentioned earlier in the previous section, Arthur found an orthonormal basis
for the space of supertempered distributions consisting of
the virtual characters of tempered representations.
It is clear from the orthogonality relation (Coroallary \ref{orthogonality}) and
the above theorems of Harish-Chandra (Theorems 8.3-8.5)
that the Arthur's basis consists of characters of discrete series and appropriate
linear combinations of the characters of limits of discrete series with the same
Dirac index up to a sign.  We summarize it as the following corollary.

\begin{cor}
The characters of discrete series and the appropriate linear combinations of the characters of
limited of discrete series with the same Dirac index (up to a sign) form an orthonormal
basis of the space of supertempered distributions.
\end{cor}

%
\section{Elliptic representations with regular infinitesimal characters}
%
In this section we still assume that $G(\R)\supset K(\R)$  is of equal rank
and $T(\R)$ is a maximal torus for $K(\R)$.  We consider the case that
$X$ is a simple Harish-Chandra module with regular infinitesimal character.

\begin{theorem} Suppose that an irreducible Harish-Chandra module $X$ has regular infinitesimal character.
Then we have \begin{equation}\label{Kparity}
\Hom_{\widetilde K}(H_D^+(X),H_D^-(X))=0.
\end{equation}
In particular, it follows that the Dirac index $I(X)=0$ (equivalently, its character $\theta_X=0$) if and only if
the Dirac cohomology $H_D(X)=0$.
\end{theorem}

\begin{proof}
Let $\frb=\t+\fru$ be a $\theta$-stable Borel subalgebra.
Then $\frt$ is contained in $\frk$.  We need to consider
the Dirac cohomology $H_D(\frg,\frt;X)$ of $X$ with respect to
the cubic Dirac operator $D(\frg,\frt)$.
By calculation in stages (see Section 6), we have
$$H_D(\frg,\frt;X)=H_D(\frk,\frt;H_D(X)).$$
It follows that
$$H_D^+(\frg,\frt;X)\supseteq H_D^+(\frk,\frt; H_D^+(X))$$
and
$$H_D^-(\frg,\frt;X)\supseteq H_D^+(\frk,\frt; H_D^-(X)).$$
Clearly, the following condition
\begin{equation}\label{Tparity}
\Hom_{\widetilde T}(H_D^+(\frg,\frt;X),H_D^-(\frg,\frt;X))=0
\end{equation}
implies (\ref{Kparity}).   It remains to prove (\ref{Tparity}).  We note that
it follows from a  theorem of Vogan (Theorem 7.2 [V1]) that
$$\Hom_T (H^+(\fru,X),H^-(\fru,X))=0,$$
for any irreducible Harish-Chandra module $X$ with regular infinitesimal character.
Then (\ref{Tparity}) follows from the above parity condition on $\fru$-cohomology
if we have the following embeddings
\begin{equation}\label{embed}H^\pm_D(\frg,\frt;X)\subseteq H^\pm(\fru,X)\otimes Z_{\rho(\bar \fru)}.
\end{equation}
Indeed, this can be done with slightly more deliberation by using the same argument as in Proposition 4.8 [HX].
There are only finitely
many $\Kt$-types in $X\otimes S$ that can possibly contribute to
the Dirac cohomology with respect to $D(\frg,\frk)$, and therefore also finitely many to
the Dirac cohomology with respect to
$$D=D(\frg,\frt)=D(\frg,\frk)+D_\Delta(\frk,\frt)$$ by calculation in stages.
Recall in the proof of Proposition 4.8 [HX] one has $D=d+2\partial$
and $D^2=2\partial d+2d\partial$ and the decomposition
$$X\otimes S = \Ker D^2 \oplus \im D^2$$
where $\partial$ and $d$ are the differentials for $\fru$-homology and $\fru$-cohomology.
We note that an irreducible $(\frg,\frk)$-module may not be $\frt$-admissible,
and $\Ker D^2$ can be infinite-dimensional.  To make the argument in Proposition
4.8 [HX] still works in this case, we consider
$${\widetilde D}(\frg,\frt)
=D(\frg,\frk)+iD_\Delta(\frk,\frt)$$
as in Remark 6.6.
It follows from $[{\widetilde D},D^2]=0$ that $\Ker D^2$ is stable under
the action of $\widetilde D$. We restrict $\widetilde D$
to $\Ker D^2$ and have the following decomposition
$$\Ker D^2=(\ker {\widetilde D}^2\cap\Ker D^2)\oplus (\im {\widetilde D}^2\cap \Ker D^2).$$
It is clear that $U=\Ker {\widetilde D}^2\cap \Ker D^2=\Ker D(\frg,\frk)^2\cap \Ker D(\frk,\frt)^2$
is finite-dimensional.   We set
$$V=U\oplus \partial U\oplus
dU\oplus +\partial dU.$$
Then $V$ is finite-dimensional.
It follows from  $\partial d=-d\partial$ on $\Ker D^2$  that
$V$ is stable under the action of $\partial$ and $d$, as well as $D=d+2\partial$.
If we replace $\Ker D^2$ by $V$ in the final step of the proof of Proposition 4.8 [HX],
then the same argument works here.
Precisely, we restrict all operators $D$, $\partial$ and $d$ to $V$.
We have $D^2=0$ and thus  $\im D\subset \Ker D$.
Note that $\Ker D^2/\Ker D\simeq\im D$. Denote by $\partial'$ the
map of $\partial$ restrict to $V$. Then $\Ker D^2/\Ker
\partial'\simeq\im \partial'$. Recall that $\ch$ denotes the formal
$\frt$-characters.  We obtain
\begin{equation*}
\ch \im D+\ch \Ker D=\ch \im \partial'+\ch \Ker \partial'.
\end{equation*}
Moreover, one has $\im \partial'\subseteq\Ker \partial'$ since
$\partial^2=0$. Therefore,
\begin{equation}\label{chD equation1}
\ch \Ker \partial'/\im \partial'-\ch \Ker D/\im D =2(\ch \Ker
\partial'-\ch \Ker D).
\end{equation}
Then all the modules here are direct sum of
finite dimensional  $\frt$-modules. It follows from Lemma 4.6 of [HX]
and (\ref{chD equation1}) there is an
injective $\frt$-module homomorphism
\begin{equation*}
\Ker D/\im D\rightarrow \Ker \partial'/\im \partial'.
\end{equation*}
Note that the right side can be naturally embedded into $\Ker
\partial/\im \partial$. This gives the embedding of Dirac cohomology into
$\fru$-homology and we get the embedding into $\fru$-cohomology similarly.
\end{proof}

We note that in the above proof we conclude that the embeddings (\ref{embed})
are actually isomorphisms
$$H^\pm_D(\frg,\frt;X)\cong H^\pm(\fru,X)\otimes Z_{\rho(\bar \fru)}.$$

\begin{remark}
We remark that the parity condition $\Hom_T (H^+(\fru,X),H^-(\fru,X))=0$
may fail if infinitesimal character of $X$ is not regular.
I learned the following example from the lecture by Wilfred Schmid
at the Vogan Conference at MIT (in May of 2014)
and the lecture note by Dragan Mili\v ci\'c at a recent conference
(in summer of 2014) at Jacobs University in Bremen.
Let $G$ be $SU(2,1)$  and $\frb$ a $\theta$-stable parabolic which contains
neither $\frp^+$ nor $\frp^-$. Let $X$ be the degenerate limit
of discrete series with infinitesimal character $0$.
Then $X$ fails to satisfy the parity condition.
\end{remark}

We note that in the above example,  if $X$ is the limit of discrete series of $G=SU(2,1)$ with
infinitesimal character $0$, then the Dirac cohomology of $X$ is zero and the embeddings
$$H^\pm_D(\frg,\frt;X)\subseteq H^\pm(\fru,X)\otimes Z_{\rho(\bar \fru)}$$
are not isomorphisms.  However, the parity condition for Dirac cohomology is still true.
All examples we know indicate that this is perhaps true in general.

\begin{conj}  Let $X$ be an irreducible $(\frg,K)$-module. Then
$$\Hom_{\widetilde K}(H_D^+(X),H_D^-(X))=0.$$
\end{conj}
As we have already mentioned in the previous section,
the above conjecture is true if $X$ is a tempered Harish-Chandra module.

It is a natural question to classify irreducible unitary elliptic representations of
$G(\R)$ and to classify irreducible unitary representations with nonzero Dirac cohomology.
We can solve this problem under the condition that the infinitesimal character is regular.
We first recall a theorem of Salamanca-Riba.
\begin{theorem}[Salamanca-Riba \cite{SR}]  Let $G$ be a connected reductive Lie group.
If $X$ is an irreducible unitary
$(\frg,K)$-module with strongly regular infinitesimal character,
then $X\cong A_\frq(\lambda)$ for certain $\theta$-stable parabolic $\frq$ and
$\lambda$.
\end{theorem}

\begin{theorem}  Suppose $\pi$ is an irreducible unitary elliptic representation of
$G(\R)$ with a regular infinitesimal character.
Then $X_\pi\cong A_\frq(\lambda)$.
\end{theorem}

\begin{proof}
Since $X_\pi$ has nonzero Dirac cohomology, its infinitesimal character
is analytically integral for $K(\R)$ as well as for a compact real form
of $G(\bbC)$, and hence it is integral in $\Delta(\frg,\frt)$.
Then the regular infinitesimal character of $X_\pi$ is strongly
regular, and $X_\pi\cong A_\frq(\lambda)$ follows from Salamanca-Riba's theorem.
\end{proof}

Suppose that $\pi$ is an irreducible unitary representation.
It is a natural question to ask to what extent the Dirac cohomology $H_D(X_\pi)$
determines the representation $\pi$ itself.
For representations with singular
infinitesimal characters, it is easy to give examples of two non-isomorphic limits of discrete series
$\pi_1$ and $\pi_2$ such that $H_D(X_{\pi_1})=H_D(X_{\pi_2})$.
The above theorem says under the condition of
regular infinitesimal character, the question is reduced to
 $A_\frq(\lambda)$-modules.  Now the question is: if two unitary $A_\frq(\lambda)$-modules
 have isomorphic Dirac cohomology, would these two modules be isomorphic?
 The answer is not always affirmative. For example, when $G$ is $SO(2n,1)$,
 there are many non-isomorphic $A_\frq(\lambda)$-modules
that have the isomorphic Dirac cohomology.

%
\section{Pseudo-coefficients of discrete series}
%

Many important questions on non-commutative Lie groups boil
down to questions in invariant harmonic analysis:
the study of distributions on groups that are invariant under conjugacy.
The fundamental objects of invariant harmonic analysis are orbital integrals
as the geometric objects and characters of representations as the spectral objects.
The correspondence of these two kinds of objects reflects the core idea of harmonic analysis.

The orbital integrals are parameterized by the set of regular semisimple conjugacy classes in $G$.
Recall for such a $\gamma$, the orbital integral is defined as
$$\caO_\gamma(f)=\int_{G/G_\gamma}f(x^{-1}\gamma x) dx, \ \ f\in C^\infty_c(G),$$
and the stable  orbital integral is defined as
$$S\caO_\gamma(f)=\sum_{\gamma'\in S(\gamma)}\caO_{\gamma'}(f),$$
where $S(\gamma)$ is the stable conjugacy class.

Let $1\!\!1$ denote the trivial representation of $G$ and $\theta_{1\!\!1}$
the character of the Dirac index of the trivial representation.  That is
$$\theta_{1\!\!1} = \ch H_D^+(1\!\!1)-\ch H_D^-(1\!\!1)=\ch S^+ -\ch S^-.$$
We note that
$$\overline{\theta_{1\!\!1}} =(-1)^q  (\ch S^+ -\ch S^-)=(-1)^q \theta_{1\!\!1},$$
where $q={1\over 2} \dim G(\R)/K(\R)$.

Recall that $\theta_{\pi}$ denotes the character of the Dirac
index of $\pi$.  If $\pi$ is the discrete series representation with Dirac cohomology $E_\mu$,
then $$\theta_{\pi}= (-1)^q  \chi_\mu.$$
Labesse showed that there exists a function $f_\pi$ so that for any admissible
representations $\pi'$,
$$\tr \pi'(f_\pi)=\int_K \Theta_{\pi'}(k)\overline{\theta_{1\!\!1}\cdot \theta_{\pi}}dk.$$

Let $\pi'$ be a discrete series representation with Dirac cohomology $E_{\mu'}$.
It follows that
$$\tr \pi'(f_\pi)=(\chi_{\mu'},\chi_{\mu})=\dim \Hom_K(E_{\mu'},E_{\mu}).$$
Then we have the following theorem.

\begin{theorem}[Labesse \cite{Lab1}]
The function $f_\pi$ is a pseudo-coefficient for the discrete series $\pi$,
i.e., for any irreducible tempered representation $\pi'$,
$$\tr \pi'(f_\pi)=\begin{cases}1 \text{\ \ \ if \ }\pi\cong \pi'\\
                                                        0\text{\ \ \ \ otherwise}.\\  \end{cases} $$
\end{theorem}

\begin{remark}\label{character}
The orbital integrals of the pseudo-coefficient $f_\pi$ are easily computed for $\gamma$ regular semisimple:
$$\caO_\gamma(f_\pi)=\begin{cases}\Theta_\pi(\gamma^{-1}) \text{\ \ \ if  $\gamma$ is elliptic}\\
                                                        0\text{\ \ \ \ \ \ \ \ \ \ \  if $\gamma$ is not elliptic}.\\  \end{cases} $$
\end{remark}

\section{Endoscopic transfer }
%

In the Langlands program a cruder form of conjugacy called stable conjugacy plays an important role.
The study of Langlands functoriality often leads to correspondence that is defined only up to stable conjugacy.
The endoscopy theory investigates the difference between ordinary and stable conjugacy
and how to understand ordinary conjugacy inside stable conjugacy.
The aim is to recover orbital integrals and characters from endoscopy groups.

Recall that $G$ is a connected reductive algebraic group defined over $\bbR$.
Denote by $G^\vee$ the complex dual group and ${}^LG$ the $L$-group which is
the semidirect product of $G^\vee$ and the Weil group $W_\bbR$.  A Langlands
parameter is an $L$-homomorphism
$$\phi\colon W_\bbR\rightarrow {}^LG.$$
Two Langlands parameters are equivalent if they are conjugated by an inner automorphism of $G^\vee$.
An equivalence class of Langlands parameters is associated
to a packet of irreducible admissible representations of
$G(\bbR)$ [L2].  The $L$-packets  of Langlands parameters with bounded image
consist of tempered representations. Temperedness is respected by
$L$-packets, but not unitarity.

The discrete series $L$-packets are in bijection with the irreducible finite-dimensional
 representations of the same infinitesimal character.
 One can construct all tempered
 irreducible representations using unitary parabolic induction and by taking subrepresentations.
 Two tempered irreducible representations $\pi$ and $\pi'$ are in the same
 $L$-packet if up to equivalence, $\pi$ and $\pi'$ are subrepresentations of parabolically
 induced representations from discrete series $\sigma$ and $\sigma'$ in the same
 $L$-packets.

 A stable distribution is any element of the closure of the space spanned by all distributions
 of the form $\sum_{\pi\in \Pi}\Theta_\pi$ for $\Pi$ any tempered $L$-packet.
 Such distributions can be transferred to inner forms of $G$ via the matching
 of the stable orbital integrals, while unstable distributions cannot be.

 For non-tempered case  we need the Arthur packets, which are parameterized by mappings
$$\psi\colon  W_\bbR \times SL(2,\bbC) \rightarrow {}^LG$$
for which the projection onto the dual group $G^\vee$ of $\psi(W_\bbR)$ is relatively compact.
Adams and Johnson [AJ] have constructed some $A$-packets consisting of unitary $A_\frq(\lambda)$-modules.
The determination of Dirac cohomology of
$A_\frq(\lambda)$-modules may have some bearing on answering Arthur's questions (See Section 9 of [A2])
on Arthur packet $\Pi_\psi$.

In the setting of endoscopy embedding
$$\xi: {}^LH\rightarrow {}^LG,$$
one has a map from Langlands parameters for $H$ to that for $G$.
The Langlands functoriality principle asserts that there should be a
map from the Grothendieck group of virtual representations of $H(\bbR)$ to
that of $G(\bbR)$, compatible with $L$-packets.

The endoscopy theory for real groups is established by Shelstad in a series of papers
[Sh1-5].
Recasting Shelstad's work explicitly
in terms of the general transfer factors defined later by Langlands and
Shelstad [LS] is the first of the `Problems for Real Groups' proposed by Arthur [A3].

We follow Labesse \S6.7 \cite{Lab2} for the description of the endoscopic transfer.
Let  $T$ be an elliptic torus of $G$ and
 $\kappa$ an endoscopic character. Let $H$ be the endoscopic
 group defined by $(T,\kappa)$.  Let $B_G$ be a Borel subgroup of $G$ containing $T$.  Set
 $$\Delta_B(\gamma)=\Pi_{\alpha>0}(1-\gamma^{-\alpha}),$$
 where the product is over the positive roots defined by $B$.  There is only one choice
 of a Borel subgroup $B_H$ in $H$, containing $T_H$ and compatible with the isomorphism
 $j\colon T_H\cong T$.

Assume $\eta\colon {}^LH\rightarrow {}^LG$ is an admissible embedding (see \S6.6 \cite{Lab2}).
Then for any pseudo-coefficent $f$ of a discrete series of $G$, there is a linear combination
$f^H$ of pseudo-coefficents of discrete series of $H$ such that for $\gamma=j(\gamma_H)$ regular
in $T(\R)$ (see Prop. 6.7.1 \cite{Lab2}), one has
\begin{equation}\label{endoscopy}
\caS\caO_{\gamma_H}(f^H)=\Delta(\gamma_H,\gamma_G)\caO^\kappa_{\gamma_G}(f),
\end{equation}
 where the transfer factor
\begin{equation}\label{transfer}
\Delta(\gamma_H,\gamma_G)=(-1)^{q(G)-q(H)}\chi_{G,H}(\gamma)\Delta_B(\gamma^{-1})\Delta_{B_H}(\gamma_H^{-1})^{-1}.
\end{equation}

 The transfer $f\mapsto f^H$ of the pseudo-coefficents of discrete series can be extended to all of
functions in $C^\infty_c(G(\R))$ with extension of the correspondence $\gamma\mapsto \gamma_H$ (see
Theorem 6.7.2 \cite{Lab2}) so that the above identity (\ref{endoscopy}) holds for all $f$.

The geometric transfer $f\mapsto f^H$ is dual of a transfer for representations.
Given any admissible irreducible representation $\sigma$ of $H(\R)$, it corresponds to
an element $\sigma_G$ in the Grothendieck group of virtual representations of $G(\R)$ as follows.
Let $\phi$ be the langlands parameter for $\sigma$.
Let $\Sigma$ be the $L$-packet of the admissible irreducible representations
of $H(\R)$ corresponding to a Langlands parameter $\phi$ and $\Pi$ the L-packet of representations of
$G(\R)$ corresponding to $\eta \circ \phi$ (that can be an empty set if this parameter
is not relevant for $G$).

\begin{theorem} [Theorem 4.1.1 \cite{S}, Theorem 6.7.3 \cite{Lab2}] There is a function
$$\epsilon\colon \Pi \rightarrow \pm 1$$
such that, if we consider $\sigma_G$ in the Grothendieck group defined by
$$\sigma_G=\sum_{\pi\in\Pi}\epsilon(\pi)\pi$$
 then the transfer $\sigma\mapsto \sigma_G$ satisfies
 $$\tr \sigma_G(f)=\tr \sigma(f^H).$$
 \end{theorem}

In the following we suppose that $G(\R)$ has a compact maximal torus $T(\R)$,
and $\rho-\rho_H$ the difference of half sum of positive
roots for $G$ and $H$ respectively, defines a character of $T(\R)$.
In  \S 7.2 of [Lab2] Labesse shows that  the canonical
 transfer factor:
 $$\Delta(\gamma^{-1})=(-1)^{q(G)-q(H)}
 \frac{\sum_{w\in W(\frg)}\epsilon (w)\gamma^{w\rho}}{\sum_{w\in W(\frh)}\epsilon (w)\gamma^{w\rho_H}}$$
 is well-defined function.
 Then the transfer factor can be expressed more explicitly if $H$ is a subgroup of $G$.
Suppose that $\frg=\frh\oplus \frs$ is the orthogonal decomposition with
respect to a non-degenerate invariant bilinear form so that the form is non-degenerate on $\frs$.
 We write $S(\frg/\frh)$ for the spin-module of the Clifford algebra $C(\frs)$.  Then
$$\Delta(\gamma^{-1})=\ch S^+(\frg/\frh) -\ch S^-(\frg/\frh).$$
In other words,  $\Delta(\gamma^{-1})$ is equal to the character of the Dirac index of
the trivial representation with respect to the Dirac operator $D(\frg,\frh)$.
 If $\Theta_\pi$ is the character of a finite-dimensional representation $\pi$,
 then
 $$\Delta(\gamma^{-1})\Theta_\pi$$ is the character of the Dirac index of $\pi$.
 This character can be calculated easily from the Kostant formula in Section 5.
 We denote by $F_\lambda$ the irreducible finite-dimensional representation of $G(\R)$
 with highest weight $\lambda$ and by $E_\mu$ irreducible  finite-dimensional representation
 of $H(\R)$ with highest weight $\mu$. Then
 $$\Delta(\gamma^{-1})\Theta_{F_\lambda}= \sum_{w\in W^1}\Theta_{E_w(\lambda+\rho)-\rho_\frh}.$$
 Here $W^1$ is a subset of elements in $W$ corresponding to $W_\frh \backslash W$ as before.

 In view of Remark \ref{character}, the right hand side of (\ref{endoscopy})
 is the Dirac index of a combination of discrete series of $G(\R)$
 and the left hand side is a linear combination of discrete series of $H(\R)$.
 It follows from the Harish-Chandra formula for the character of discrete series
 and supertempered distributions (see Theorem \ref{super tempered}) that
 the Dirac index of a discrete series $\pi_\lambda$ with Harish-Chandra parameter $\lambda$ is
 $$\Delta(\gamma^{-1})\Theta{\pi_\lambda}=\sum_{w\in W_{K}^1} \text{sign}(w)\Theta_{\tau_{w\lambda}}.$$
Here ${\tau_{w\lambda}}$ denotes the discrete series for $H(\R)$ with Harish-Chandra parameter $w\lambda$, and
$W_{K}^1$ is a subset of elements in $W_{K}$ corresponding to $W_{H\cap K} \backslash W_{K}$.
This calculation is compatible with Labesse's calculation of the transfer of the pseudo-coefficients
of discrete series in \S 7.2 \cite{Lab2}.

The above interpretation of the transfer factors in certain cases of endoscopy
as the difference of the even and odd parts of the spin modules is clearly useful for
calculation.  It is also  reminiscent of the transfer factors for the metaplectic groups, which is given by
the formal difference of the metaplectic representations, in  the work by
Jeff Adams [A], David Renard [R1] and Wen-Wei Li [Li].  It is worthwhile investigating the
Dirac cohomology and Dirac index with respect to the symplectic Dirac operators in connection with
the Weyl algebras and the oscillator representations of metapletic groups.

\section{Hypoelliptic representations}
%
In this final section we assume that $G(\R)\supset K(\R)$ is not necessarily of equal rank.
If $G(\R)$ is indeed not of equal rank, then there is no elliptic representation for $G(\R)$.
Still, we know $G(\R)$ has
representations with nonzero Dirac cohomology.  The natural generalization of
the concept of elliptic representation for unequal rank $G(\R)$ is the following.

\begin{definition}A representation is called {\it hypoelliptic} if its global character is not identically zero  on the set of regular elements in
a fundamental Cartan subgroup.
\end{definition}
By definition, an elliptic representation is hypoelliptic.

It is a natural question to ask the relationship between hypoelliptic representations
and representations with nonzero Dirac cohomology.

\begin{conj} Suppose that $\pi$ is an irreducible admissible representation.
Then  $H_D(X_\pi)\neq  0$ implies that $\pi$ is hypoelliptic.
\end{conj}

Recall that if  $G(\R)$ is of equal rank with
$K(\R)$ then an irreducible tempered representation is either
elliptic or induced from a tempered elliptic representation by parabolic induction.

\begin{conj}  A unitary representation is either having nonzero Dirac cohomology
or induced from a unitary representation with nonzero Dirac cohomology by parabolic induction.
\end{conj}

The above conjecture holds  for $GL(n,\R)$, $GL(n,\bbC)$, $GL(n,\bbH)$
as well as $\widetilde{GL}(n,\R)$ (the two-fold covering group of $GL(n,\R)$).

A recent preprint of Adams-van Leeuwen-Trapa-Vogan \cite{ALTV} gives an algorithm to determine
the irreducible unitary representations.  The  above conjecture means that one may regard
unitary representations with nonzero Dirac cohomology as  `cuspidal' ones.

\section*{Acknowledgement}  It is a great pleasure to thank David Vogan and Pavle Pand\v zi\'c
for the helpful discussions.


\end{document}